\documentclass[12pt,final,reqno]{amsart}
\usepackage{amsthm}
\numberwithin{equation}{section}
\allowdisplaybreaks[1]
\usepackage{amsmath,amsfonts,amssymb,mathtools}
\usepackage[utf8]{inputenc}
\usepackage[T1]{fontenc}
\RequirePackage{ifthen,xifthen}
\usepackage[shortlabels]{enumitem}
\usepackage[ linktocpage
           , pdftex
           , colorlinks=false
           , pdfborder={0 0 0}
           , final
           ]{hyperref}
\usepackage{tikz}
\usetikzlibrary{arrows,positioning}
\usetikzlibrary{patterns}
\usetikzlibrary{automata,er}
\usetikzlibrary{positioning,decorations.pathreplacing,shapes.multipart}
\usetikzlibrary{calc}

\usepackage[ngerman,main=english]{babel}\usepackage{csquotes}
\usepackage[style=alphabetic 
           ,giveninits
           ,backref=true 
           ,maxnames=10
           ,mincrossrefs=5
           ,backend=biber
           ]{biblatex}

\AtEveryBibitem{\clearfield{month}}
\AtEveryBibitem{\clearfield{day}}
\renewbibmacro{in:}{
  \ifentrytype{article}
    {}
    {\printtext{\bibstring{in}\intitlepunct}}%
}
\renewbibmacro*{doi+eprint+url}{
  \iftoggle{bbx:doi}
    {\printfield{doi}}
    {}%
  \newunit\newblock
  \iftoggle{bbx:eprint}
    {\usebibmacro{eprint}}
    {}%
  \newunit\newblock
  \iftoggle{bbx:url}
    {\iffieldundef{doi}
      {\usebibmacro{url+urldate}}
      {\clearfield{url}}%
    }
    {}%
}
\DeclareFieldFormat{eprint:euclid}{Project Euclid:
  \ifhyperref
    {\href{https://projecteuclid.org/euclid.#1}{#1}}
    {https://projecteuclid.org/euclid.#1}%
}
\DeclareFieldFormat{isbn}{%
  \mkbibacro{ISBN}\addcolon\space
  \ifhyperref
    {\href{https://en.wikipedia.org/w/index.php?title=Special\%3ABookSources\&isbn=#1}{\nolinkurl{#1}}}
    {\nolinkurl{#1}}%
}

\usepackage[capitalise,nameinlink]{cleveref}
\Crefname{figure}{Figure}{Figures}
\usepackage{aliascnt}          
\newcommand\mynewtheorem[4][]{ 
  \ifthenelse{\isempty{#1}}{   
    \newtheorem{#2}{#3}        
  }{
    \newaliascnt{#2}{#1}       
    \newtheorem{#2}[#2]{#3}    
    \aliascntresetthe{#2}      
  }
  \crefname{#2}{#3}{#4}        
}
\newtheorem{thm}{Theorem}[section]\crefname{thm}{Theorem}{Theorems}
\mynewtheorem[thm]{prop}{Proposition}{Propositions}
\mynewtheorem[thm]{cor}{Corollary}{Corollaries}
\mynewtheorem[thm]{lemma}{Lemma}{Lemmas}
\theoremstyle{definition}
\mynewtheorem[thm]{definition}{Definition}{Definitions}
\mynewtheorem[thm]{assumption}{Assumption}{Assumptions}
\mynewtheorem[thm]{example}{Example}{Examples}
\theoremstyle{remark}
\mynewtheorem[thm]{remark}{Remark}{Remarks}
\RequirePackage{xcolor}

\providecommand{\dom}{\operatorname*{dom}}
\providecommand{\esssup}[1][]{\operatorname*{\ifthenelse{\equal{}{#1}}{}{\mathit{#1}-}ess\,sup}}
\providecommand{\essinf}[1][]{\operatorname*{\ifthenelse{\equal{}{#1}}{}{\mathit{#1}-}ess\,inf}}
\newcommand{\supp}{\operatorname{supp}}
\newcommand{\dist}{\operatorname{dist}}
\newcommand{\norm}[2][]{\lVert#2\ifthenelse{\equal{}{#1}}{\rVert}{\mathclose{\rVert}_{#1}}}
\newcommand{\bignorm}[2][]%
  {\bigl\lVert#2\bigr\rVert\ifthenelse{\equal{}{#1}}{}{_{#1}}}
\newcommand{\Bignorm}[2][]%
  {\Bigl\lVert#2\Bigr\rVert\ifthenelse{\equal{}{#1}}{}{_{#1}}}
\newcommand{\biggnorm}[2][]%
  {\biggl\lVert#2\biggr\rVert\ifthenelse{\equal{}{#1}}{}{_{#1}}}
\newcommand{\Norm}[2][]%
  {\left\lVert#2\right\rVert\ifthenelse{\equal{}{#1}}{}{_{#1}}}
\newcommand{\spr}[2]{\langle#1,#2\rangle}
\newcommand{\Id}[1][]{\operatorname{Id\ifthenelse{\equal{}{#1}}{}{_{#1}}}}
\newcommand{\Tr}{\operatorname{Tr}}
\usepackage{bbm}
\newcommand{\E}{\mathbbm E}
\newcommand{\Prob}{\mathbbm P}
\newcommand{\bC}{\mathbbm C}
\newcommand{\bN}{\mathbbm N}
\newcommand{\bQ}{\mathbbm Q}
\newcommand{\bR}{\mathbbm R}
\newcommand{\bZ}{\mathbbm Z}
\newcommand{\cB}{\mathcal B}
\newcommand{\cD}{\mathcal D}
\newcommand{\cDc}{\cD}
\newcommand{\cF}{\mathcal F}%
\newcommand{\cH}{\mathcal H}
\newcommand{\cL}{\mathcal L}
\newcommand{\textq}[1]{\text{#1}\quad}
\newcommand{\qtextq}[1]{\quad\text{#1}\quad}
\newcommand{\qtext}[1]{\quad\text{#1}}
\newcommand{\qqtext}[1]{\qquad\text{#1}}
\let\originald\d 
\renewcommand{\d}{\ifthenelse{\boolean{mmode}}{\mathrm d}{\originald}}
\def\Int#1d{\int#1\,\d} 
\providecommand{\cc}[1]{\overline{#1}}
\providecommand{\grad}{\nabla}
\DeclareMathOperator{\divergence}{div}
\DeclarePairedDelimiter\abs\lvert\rvert
\providecommand{\bigabs}[2][]{{\bigl\lvert#2\bigr\rvert}\ifthenelse{\equal{#1}{}}{}{_{#1}}}
\DeclarePairedDelimiter\ceil\lceil\rceil
\DeclarePairedDelimiter\floor\lfloor\rfloor
\providecommand{\setsize}[1]{\operatorname{\#}#1}

\providecommand{\ve}{\varepsilon}
\providecommand{\vol}{\abs}

\providecommand{\xto}{\xrightarrow}

\providecommand{\dnto}{\searrow}
\providecommand{\ifu}[1]{\mathbf 1\ifthenelse{\equal{#1}{_}}{#1}{_{#1}}}
\providecommand{\Vper}{V_{\mathrm{per}}}
\providecommand{\Hper}{H_{\mathrm{per}}}
\DeclarePairedDelimiterX\Ioo[2](){#1,#2}
\DeclarePairedDelimiterX\Ico[2][){#1,#2}
\DeclarePairedDelimiterX\Ioc[2](]{#1,#2}
\DeclarePairedDelimiterX\Icc[2][]{#1,#2}
\providecommand{\etdef}{\phantom{:}&\rlap{:}=}
\providecommand{\isect}{\cap}\newcommand{\Isect}{\bigcap}
\newcommand{\Union}{\bigcup}
\newcommand{\Directsum}{\bigoplus}

\providecommand{\cA}{\mathcal A}
\providecommand{\cC}{\mathcal C}
\providecommand{\sse}{\subseteq}

\providecommand{\rstr}[1]{|_{#1}}
\providecommand{\Laplace}{\Delta}

\providecommand{\fa}{\forall}
\providecommand{\ex}{\exists}

\providecommand{\from}{\colon}
\providecommand{\e}{\mathrm e}
\providecommand{\argmt}{\mathchoice{{}\cdot{}}{{}\cdot{}}{{}\bullet{}}{{}\bullet{}}}
\let\originald\d 
\renewcommand{\d}[1]{\ifthenelse{\boolean{mmode}}{\mathrm d#1}{\originald}}
\providecommand{\dd}{\,\d}
\providecommand{\dx}{\d x}
\providecommand{\ddx}{\dd x}

\providecommand{\ddy}{\dd y}
\providecommand{\molly}[1]{\chi\ifthenelse{\equal{#1}{}}{}{_{#1}}}%
\providecommand{\M}{\mathcal M} 
\DeclareMathOperator\diag{diag}
\providecommand{\ie}{i.\,e.}
\providecommand{\eg}{e.\,g.}

\makeatletter 
\newif\ifdr@ft\dr@ftfalse
\relax
\newcounter{const@ntNo}
\newcommand{\reloadConst@nt}{%
  \stepcounter{const@ntNo}%
  \edef\genericConst@ntInternal{C_{\theconst@ntNo}}%
}
\newcommand{\genericConstant}[1]{%
  \reloadConst@nt%
  \expandafter\let\csname #1\endcsname\genericConst@ntInternal%
  \ifthenelse{\boolean{dr@ft}}{%
    \marginpar{\textcolor{gray}{\fbox{#1 = $\csname #1\endcsname$}}}%
  }{}%
}
\makeatother

\newcommand{\Schroedinger}{\foreignlanguage{ngerman}{Schr\"odinger}}

\begin{document}

\title[Random breather models]
  {Lifshitz asymptotics and localization for random breather models}
\author{Christoph Schumacher \and Ivan Veseli\'c}
\address{Fakult\"at f\"ur Mathematik, 44227 TU Dortmund}

\begin{abstract}
  We prove Lifshitz behavior at the bottom of the spectrum
  for non--negative random potentials,
  i.\,e.\ show that the IDS is exponentially small at low energies.
  The theory is developed for the breather potential and generalized to all
  non--negative random potentials in a second step.
  Since our models need not be ergodic, we need to identify the minimum of the spectrum.
  We deduce an initial length scale estimate from the Lifshitz bound for the breather model
  and combine it with a recent Wegner estimate
  to establishes Anderson localization via multi-scale analysis.
  Finally, for ergodic models, we complement the Lifshitz behavior with a lower bound.
  We provide detailed proofs accessible to non-experts.
\end{abstract}

\maketitle

\tableofcontents

\section{Introduction}

\subsection{Motivation}

Transport properties of disordered media continuously attract substantial attention
in physical and mathematical literature.
In the latter they are often modelled by random Schr\"odinger operators.
The corresponding theory has been developed over several decades,
and while certain questions have been successfully answered, new ones emerge.
This review focuses on two of these research directions:
The first one is interest in properties of random operators with non-linear dependence on randomness and the second one in the role of ergodicity - actually, the lack thereof - in the modelling.

Why do these to features require new ideas?

Random operators with linear dependence on random variables facilitate specialised
perturbation techniques and explicit formulas, for instance spectral averaging.
Non-linear dependence on the other hand requires more robust analytic and probabilistic tools.

The most basic challenge appearing when studying non-ergodic models is the identification, or even definition, of spectral edges,
in our particular case of the spectral minimum.
We have to resolve this question in order to formulate a meaningful Lifshitz bound.

The models we study here, namely random breather potentials,
exemplify both the non-linear dependence on randomness and the non-stationarity of the potential.
In fact, in certain aspects, the random breather model is paradigmatic for monotone random potentials:
For a large class of such potentials the proof of Lifshitz bounds can be reduced to the one for a standard breather model.

If one assumes additionally that the model is ergodic, the integrated density of states exist and
our results imply indeed Lifshits tails in the usual sense.

Subsequently, we prove localisation at the bottom of the spectrum for a class of random breather models.
This is established using the multi-scale analysis, which has been developed and perfected in a continuous effort, see for instance  \cite{Stollmann-01,GerminetK-04} and the references therein.
The induction anchor is provided by our mentioned Lifshitz tail bound and the induction step is inferred from a versatile Wegner estimate of \cite{NakicTTV-15,NakicTTV-18b}.

While we defer a broader discussion of the history of the subject to the subsequent subsection let us mention briefly that random
breather models have been studied among others in \cite{CombesHM-96,CombesHN-01,KirschV-10,NakicTTV-15,SchumacherV-17,NakicTTV-18b}.

One aim of this work is to extend the results of the paper \cite{KirschV-10},
where Lifshitz tails have been established for a restricted class of breather models.
The method used there is not applicable to the models we study here, not even to the standard breather model, cf.~Example~\ref{exmp:standard-breather}.
The reason is, that in \cite{KirschV-10} strong regularity conditions are imposed in order to linearise the breather model
and thus reduce its study to the alloy type model.

\subsection{Sketch of results}
We present here the main results of this paper in a simplified setting.

Let $C_+>0$, $u_k\in L^\infty(\bR^d)$, $k\in\bZ^d$,
be non--negative and non--vanishing with support in the ball of radius $C_+$ around the origin.
Furthermore, let $\lambda_k\colon \Omega \to \Icc0{C_+}$, $k\in\bZ^d$, be independent random variables.
Construct the random breather potential $W_\omega$ as follows
\begin{equation*}
  W\colon\Omega\times\bR^d \to \Ico0\infty\textq,
  W_\omega(x):=\sum_{k \in\bZ^d} u_k\left(\frac{x-k}{\lambda_k}\right)
\end{equation*}
where we set  $u_k(y/\lambda)=0$ for $\lambda=0$.
Now we define for a $\bZ^d$-periodic potential~$\Vper\in L^\infty(\bR^d,\bR)$
a periodic and a random Schr\"odinger operator
\begin{equation*}
  \Hper:=-\Laplace+\Vper
  \qtextq{and}
  H_\omega:=\Hper+W_\omega.
\end{equation*}

In order to have a truly random potential, we have to avoid a deterministic situation.
The following condition ensures that the potential $W_\omega$ is \emph{sometimes not too small}.
We require that there exists a $\mu>0$ such that
  \begin{equation*}
    \inf_{k\in\bZ^d}\Prob\left\{
      \vol{\{x\in\left[-\frac{1}{2},\frac{1}{2}\right]^d\mid u_k\left(\frac{x-k}{\lambda_k}\right)\ge\mu\}}\ge\mu \right\}
      \geq\mu\text.
  \end{equation*}
This implies that the $u_k$ have some support near the origin.

The following lemma ensures on the other hand that the potential~$W_\omega$ is \emph{sometimes really small}.
In particular, there is randomness in the model.
\begin{lemma}
  Set $E_0:=\inf\sigma(\Hper)$.
  Assume that for all $\kappa>0$ and all $k\in\bZ^d$,   we have
  \begin{equation}
    \Prob\left\{\biggnorm[p]{u_k\left(\frac{\cdot}{\lambda_k}\right)}
      \le\kappa\right\}
    >0\text.
  \end{equation}
  Then, the operator~$H_\omega$   satisfies for all $\varepsilon>0$
  \begin{equation*}
    \Prob\{\omega\in\Omega:\inf\sigma(H_\omega)\le E_0+\varepsilon\}>0\text.
  \end{equation*}
\end{lemma}
This identifies $E_0$ as the overall minimum of the spectrum of the non-ergodic operator $H_\omega$.

Under the assumptions spelled out so far, our (upper) Lifshitz bound now reads as follows:
\begin{thm}\label{thm:Intro-Lif}
  There exist constants $C,\delta\ge0$ and $E'>E_0$
  such that for all $E\in\Ioc{E_0}{E'}$
  \begin{equation}
    \Prob\bigl\{\omega\in\Omega:E_1(H_\omega^{L_E})\le E\bigr\}
      \le\exp\bigl(-C(E-E_0)^{-\frac d2}\bigr)\text,
  \end{equation}
  where $L_E=\floor[\big]{\sqrt{\delta/(E-E_0)}}$.
\end{thm}

In the ergodic situation, the integrated density of states (IDS) exist and exhibits Lifshitz tails.

\begin{thm}
  Let the conditions of Theorem \ref{thm:Intro-Lif} hold.
  Assume additionally that the random variables~$\lambda_k$,
  $k\in\bZ^d$, are identically distributed and $u_k=u_1$ for all $k\in\bZ^d$.
  Then the IDS $N\from\bR\to\bR$ exists, and we have
  \begin{equation}
    \lim_{E\dnto E_0}\frac{\ln(-\ln(N(E)))}{\ln(E-E_0)}
      \le-\frac{d}{2}\text.
  \end{equation}
  If, in addition
  \begin{equation*}
    \exists\ \alpha_0, \eta>0\colon\forall\ \alpha\in\Icc0{\alpha_0}\colon
    \Prob\bigl\{\norm[L^\infty]{u_{\lambda_0}}\le\alpha\bigr\}
      \ge\alpha^\eta\text.
  \end{equation*}
  holds, then the IDS satisfies
  \begin{equation}
    \lim_{E\dnto E_0}\frac{\ln(-\ln(N(E)))}{\ln(E-E_0)}
      =-\frac{d}{2}\text.
  \end{equation}
  \end{thm}

If we assume that the random variables are continuously distributed we can conclude Anderson localization for the random breather model.
Let us formulate a simple model where this holds.

\begin{example}[Standard breather model]
\label{exmp:standard-breather}
  Set $u(t,x)=\mu\ifu{tB}(x)$ for $\mu>0$ and $t\in\Icc01$
  where~$B$ denotes the ball or radius one around the origin.
  The random variables~$\lambda_k$, $k \in \bZ^d$ are i.\,i.\,d.\ and distributed according to a bounded density with support in $\Icc01$.
  Set $W_\omega(x)=\sum_{k\in\bZ^d}  u(\lambda_k,x-k)$.
\end{example}

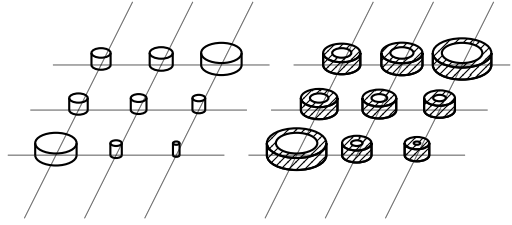
\begin{figure}\centering
\begin{tikzpicture}[scale=0.8, thick]
\pgfmathsetseed{{\number\pdfrandomseed}}

\pgfmathsetmacro{\a}{0.5};
\pgfmathsetmacro{\b}{0.75};

\pgfmathsetmacro{\d}{0.15};

\pgfmathsetmacro{\s}{0.5};

\pgfmathsetmacro{\h}{0.2};

\foreach \x in {-1,0,1}{
  \foreach \y in {-1*\b,0*\b,1*\b}{
  \pgfmathsetmacro{\r}{rand*0.16+0.20}

  \begin{scope}[xshift=0cm,yshift=0cm]
    \draw[] ({\x+\a*\y-\r},\y) arc(180:360:{\r} and {\r*\s});
    \draw[] ({\x+\a*\y-\r},\y) -- ({\x+\a*\y-\r},\y+ \h);
    \draw[] ({\x+\a*\y+\r},\y+\h) -- ({\x+\a*\y+\r},\y);
    \draw[] ({\x+\a*\y},\y+\h) ellipse ({\r} and {\r*\s});

    \draw[opacity = 0.25, ultra thin] ({\x-\a*1.8},-1.8) -- ({\x+\a*1.8},1.8);
    \draw[opacity = 0.25, ultra thin] (-1.8 + \a * \x*\b,\x*\b) -- (1.8 + \a * \x*\b,\x*\b);
  \end{scope}

    \begin{scope}[xshift=4cm,yshift=0cm]
    \draw[] ({\x+\a*\y-\r-\d},\y) arc(180:360:{\r+\d}  and {(\r+\d)*\s} );
    \draw[] ({\x+\a*\y-\r-\d},\y) -- ({\x+\a*\y-\r-\d},\y+ \h );
    \draw[] ({\x+\a*\y+\r +\d},\y+\h) -- ({\x+\a*\y+\r +\d},\y);

    \filldraw[pattern = north east lines]
      ({\x+\a*\y-\r - \d},\y)	 -> 	({\x+\a*\y-\r - \d},\y+\h)	arc(180:360:{\r+\d}  and {(\r+\d)*\s}) -- ({\x+\a*\y+\r + \d},\y) arc (360:180:{\r+\d}  and {(\r+\d)*\s});
    \filldraw[pattern = north east lines, even odd rule]
      ({\x+\a*\y},\y+\h) ellipse ({\r+\d}  and {(\r+\d)*\s} )
      ({\x+\a*\y},\y+\h) ellipse (\r cm and \r * \s cm);
    \draw[opacity = 0.25, ultra thin] ({\x-\a*1.8},-1.8) -- ({\x+\a*1.8},1.8);
    \draw[opacity = 0.25, ultra thin] (-1.8 + \a * \x*\b,\x*\b) -- (1.8 + \a * \x*\b,\x*\b);
    \end{scope}

 }}
\end{tikzpicture}
\caption{shows a realization of a standard breather potential $W_\omega$, and a second realization where all random variables are increased by the same amount.}
\end{figure}

\begin{thm}
  For $H_\omega:=\Hper+W_\omega$
  with~$W_\omega$ as in Example \ref{exmp:standard-breather},
  there exists $E'>\min\sigma(H_\omega)$ such that
  \begin{equation*}
    \Icc{\min\sigma(H_\omega)}{E'}\isect\sigma(H_\omega)
    \ne\emptyset
    \quad\text{ and }\quad
    \Icc{\min\sigma(H_\omega)}{E'}\isect\sigma_{c}(H_\omega)
    =\emptyset
    \text.
  \end{equation*}
\end{thm}

Our Lifshitz bounds are also used in \cite{TaeuferV-20}
to prove Anderson and dynamical localization
for alloy type potentials with minimal support conditions.

\subsection{History}
The most prominently studied random potential (on continuum space) is the alloy type model
\begin{equation*}
  V_\omega(x)
  =\sum_{k\in\bZ^d}\lambda_k(\omega)u(x-k)
\end{equation*}
for appropriate single site potentials~$u$
and i.\,i.\,d.\ random variables~$\lambda_k$, $k\in\bZ^d$.
The linear coupling of the random variables to the translated single site potentials facilitates the analysis for such models.
Indeed, pioneering work on Anderson localization in the continuum concentrated on such models
and the literature devoted to them is vast. We do not even attempt to give an account of it but refer to monographs,
e.\,g.~\cite{Stollmann-01,Veselic-08,AizenmanW-15}.
Specifically, results on Lifshitz tails for the alloy type and other random potentials are surveyed in
\cite{KirschM-07}.

More recently attention turned to models where the influence of the random variables $\lambda_k$ on the Hamiltonian is more involved.
Such models arise naturally, e.\,g.~when studying electromagnetic Schr\"odinger operators with random magnetic fields,
cf.~e.\,g.\ \cite{Ueki-94,Ueki-00,HislopK-02,KloppNNN-03,Ueki-08,Bourgain-09,ErdoesH-12a,ErdoesH-12c,ErdoesH-12b}.
Non-linear dependence enters naturally also, if the underlying geometry is random.
This is the case for random waveguide Laplacians,
see e.\,g.~\cite{KleespiesS-00,BorisovV-11,BorisovV-13,BorisovGV-16,Borisov-17,KirschKG-18,Najar-19},
Laplace-Beltrami operators with random metrics,
cf.~e.\,g.~\cite{LenzPV-04,LenzPPV-08,LenzPPV-09},
or random displacement models,
see~e.\,g.~\cite{Klopp-93,BakerLS-08,GhribiK-10,KloppLNS-12}.

Random breather models were introduced in \cite{CombesHM-96} and a
non-void Wegner estimate was proven in \cite{CombesHN-01}.
However, regularity conditions that have been imposed in the literature on
breather potentials are tricky and sometimes lead to empty theorems,
see the discussion and Lemma A.5 in the Appendix of \cite{NakicTTV-18b}.
In the last mentioned paper Wegner estimates are proven for a
quite wide class of random breather models.

As mentioned before, Lifshitz tails for breather potentials have been studied before in \cite{KirschV-10}. The proof mimics to a large extent the proof for alloy type models
and for this requires restrictive regularity assumptions, excluding in particular
very natural scenarios, e.g.~the standard breather model spelled out in Example \ref{exmp:standard-breather}. This particularly simple and appealing model has been
treated in \cite{Veselic-07,SchumacherV-17} in one-dimensional and multi-dimensional space,
respectively.

In a parallel development, various authors started to study random operators depending on a sequence of
variables $\lambda_k$, $k \in \bZ^d$, which are not ergodic with respect to the action of the group of translations
$\bZ^d$. This includes the works
\cite{RojasMolina-12,RojasMolinaV-13,Klein-13,GerminetMRM-15,NakicTTV-18b,MuellerRM-20}.

Finally, let us compare the present paper with
\cite{KirschM-83a}. This work studies the approximability of the IDS by a macroscopic limit and
its Lifshitz behaviour at the bottom of the spectrum using the theory of Large Deviations.
The class of potentials studied there is more abstractly described than our models.
Some of the theorems spelled out in \cite{KirschM-83a} allow for potentials with long range correlations,
merely a mixing condition is required, cf.~assumption (B) there.
On the other hand, \cite{KirschM-83a} does not allow for a periodic background potential and non-ergodic models.
Theorem 4 and 5 there provide an upper Lifshitz bound the mentioned wide class of (correlated) random fields at energy $E=0$.
However, it is not established that this energy is indeed the bottom of the spectrum, which we provide in our Lemma \ref{upperbound}.

Our proof is inspired by \cite{KirschM-83a} regarding the use of Thirring's inequality.
However, our proof is much more explicit, which is for instance manifested in the use of quantitative concentration inequalities
rather than Large Deviations Theory.

The rest of the paper is organized as follows:
We define our general model in \S \ref{s:general-setup},
prove upper Lifshitz bounds for breather potentials in \S \ref{s:breather},
extend the proof to the general class of potentials we consider in \S \ref{s:reduction},
provide complementary lower bounds in \S \ref{s:lower-IDS-bound},
and complete the paper with a \S \ref{s:localization}
on localization for a class of breather models.

\section{General setup}
\label{s:general-setup}
\subsection{Random \Schroedinger\ operators and the IDS}\label{subsect:IntroIDS}

We consider \Schroedinger\ operators on $L^2(\bR^d)$
with a random, $\cL$-ergodic potential for a cocompact lattice $\cL=\M\bZ^d$
with a real invertible $d\times d$ matrix~$\M$ with positive determinant.
More precisely, we fix a measurable space~$(\Omega_0,\cA_0)$
and a jointly measurable \emph{single site potential}
$u\from\Omega_0\times\bR^d\to\bR$.
With the notiation $u_\lambda:=u(\lambda,\argmt)\from\bR^d\to\bR$,
$\lambda\in\Omega_0$, it becomes obvious that~$\Omega_0$
serves as an index set for a whole family of single site potentials.

We combine the single site potentials randomly on~$\bR^d$.
To this end, we use the canonical probability space $(\Omega,\cA,\Prob)$
with $\Omega:=\bigotimes_{\cL}\Omega_0$
and an i.\,i.\,d.\ family of random variables
$\lambda_k\from\Omega\to\Omega_0$, indexed by $k\in\cL$, via
\begin{equation}\label{e-random-potential}
  W_\omega\from\bR^d\to\bR\textq,
  W_\omega(x)=\sum_{k\in\cL}u_{\lambda_k(\omega)}(x-k)
  \qquad(\omega\in\Omega)\text.
\end{equation}

We assume that there exists $p>\max\{2,d/2\}$ such that
$W_\omega\in L_{\mathrm{loc,unif}}^p(\bR^d)$ uniformly in~$\omega$, \ie
\begin{equation}\label{eq:Wlocp}
  \adjustlimits\sup_{\omega\in\Omega}\sup_{x\in\bR^d}\norm[p]{W_\omega\ifu{\cD+x}}
  <\infty\text,
\end{equation}
where $\cD:=\M\Ioo{-\frac12}{\frac12}^d$.
By an application of the Kato--Rellich theorem,
see \eg\ \cite[Theorem~XIII.96]{ReedS-78},
for an $\cL$-periodic potential~$\Vper\in L^\infty(\bR^d,\bR)$,
the operators
\begin{equation}\label{e-Hper-Homega}
  \Hper:=-\Laplace+\Vper
  \qtextq{and}
  H_\omega:=\Hper+W_\omega
\end{equation}
are self-adjoint on the domain $\dom(\Laplace)$ of~$\Laplace$
and lower bounded uniformly in~$\omega\in\Omega$.

The joint measurability of the single site potential implies
the measurablilty of the operator family $(H_\omega)_{\omega\in\Omega}$,
cf.~\cite{Kirsch-Martinelli-1982-Crelle}.
Moreover, $(H_\omega)_{\omega\in\Omega}$
forms an \emph{ergodic family of operators} in the following sense.
There is an ergodic $\cL$-action $\vartheta\from\cL\times\Omega\to\Omega$
on $(\Omega,\cA,\Prob)$, which satisfies
\begin{equation*}
  H_{\vartheta(x,\omega)}=U(x)^{-1}H_\omega U(x)
  \qquad(x\in\cL,\omega\in\Omega)\text,
\end{equation*}
where $U\from\cL\to\cB(L^2(\bR^d))$, $(U(z)f)(x)=f(x+z)$,
is the unitary representation of~$\cL$ on $L^2(\bR^d)$ acting by translation.
For any ergodic operator family there exists a closed set $\Sigma\subseteq\bR$
and an event $\Omega'\in\cA$ of full probability, such that for all
$\omega\in\Omega'$, the spectrum of~$H_\omega$ coincides with~$\Sigma$,
cf.~\cite{Kirsch-Martinelli-1982-Crelle}.

For the definition of the integrated density of states (IDS)
$N\from\bR\to\bR$ for $(H_\omega)_{\omega\in\Omega}$ we follow
\cite{Pastur-80,Kirsch-Martinelli-1982-JPhysA}.
Denote for $L>0$ and $x\in\bR^d$
\begin{multline*}
  \Lambda_L:=\M\Ioo[\big]{-(L+\tfrac12)}{L+\tfrac12}^d\textq,
  \Lambda_L(x):=\Lambda_L+x\text{, and}\\
  \cF:=\{\Lambda_L(x)\mid x\in\cL,L\in\bN\}\text.
\end{multline*}
Neumann~($N$), Dirichlet~($D$), periodic~($P$) and Mezincescu~($M$)
boundary conditions, a specific choice of Robin boundary conditions,
see \cref{sectionboundaryconditions},
give rise to self-adjoint restrictions $H^{\Lambda,\sharp}$,
$\sharp\in\{N,D,P,M\}$, of a self-adjoint operator~$H$ on~$\bR^d$
to the region~$\Lambda\in\cF$.
Due to~\eqref{eq:Wlocp},
the Kato-Rellich Theorem, see \eg\ \cite[Theorem~X.12]{ReedS-75},
shows that the domain of selfadjointness of~$H^{\Lambda,\sharp}$
is the same as the one of~$\Laplace^{\Lambda,\sharp}$
for all $\sharp\in\{D,N,P,M\}$.
We defer detailed definitions to \cref{sectionboundaryconditions}.
We abbreviate $\Laplace^{L,\sharp}:=\Laplace^{\Lambda_L,\sharp}$
and $H^{L,\sharp}:=H^{\Lambda_L,\sharp}$.

It is well known, see \cite{Kirsch-Martinelli-1982-JPhysA,Veselic-08},
that the finite volume restrictions of~$H_\omega$, $\omega\in\Omega$,
have compact resolvents, so that their spectrum is purely discrete.
The eigenvalue counting functions
\begin{equation*}
  n^{\sharp}(E,H_\omega,\Lambda)
    :=\Tr\bigl(\ifu{\Ioc{-\infty}E}(H_\omega^{\Lambda,\sharp})\bigr)
\end{equation*}
and its normalized versions
\begin{equation*}
  N^{\sharp}(E,H_\omega,\Lambda)
    :=\vol{\Lambda}^{-1}n^{\sharp}(E,H_\omega,\Lambda)
  \text,
\end{equation*}
are thereby well defined for
$\omega\in\Omega$, $E\in\bR$, $\Lambda\in\cF$ and $\sharp\in\{N,D,P,M\}$.
Again we write briefly
$n_L^{\sharp}(E,H_\omega):=n^{\sharp}(E,H_\omega,\Lambda_L)$,
$L\in\bN$, and analogously~$N_L^\sharp$.

The eigenvalue counting functions
are \emph{equivariant},
\ie\ for all $k\in\cL$, $E\in\bR$, $\Lambda\in\cF$,
$\omega\in\Omega$ and $\sharp\in\{D,N,P,M\}$ we have
\begin{equation*}
  N^\sharp(E,H_{\vartheta(k,\omega)},\Lambda)
    =N^\sharp(E,H_\omega,\Lambda+k)\text.
\end{equation*}
Moreover, as we will see in \cref{sectionboundaryconditions},
$n^N$ and $n^M$ are \emph{subadditive},
\ie\ for any $\Lambda\in\cF$
given as a finite disjoint union $\Lambda=\Union_j\Lambda_j$
of cubes $\Lambda_j\in\cF$,
\begin{equation*}
  n^\sharp(E,H_\omega,\Lambda)
    \le\sum\nolimits_jn^\sharp(E,H_\omega,\Lambda_j)
    \qquad(\sharp\in\{N,M\})
\end{equation*}
holds true.
Together with the ergodicity of $(H_\omega)_{\omega\in\Omega}$,
the superadditive ergodic theorem of \cite{AkcogluK-81} implies
that there exists for each $E\in\bR$
an event $\Omega_E\in\cA$ of probability~$1$,
such that for all $\omega\in\Omega_E$
\begin{equation*}
  \lim_{L\to\infty}N_L^\sharp(E,H_\omega)
    =\inf_{L\in\bN}\E[N_L^\sharp(E,H_{\argmt})]
    \qquad(\sharp\in\{N,M\})\text.
\end{equation*}
Analogously, $n^D$ is superadditive,
and, w.\,l.\,o.\,g.,\ for the same event~$\Omega_E$ we have
\begin{equation*}
  \lim_{L\to\infty}N_L^D(E,H_\omega)
    =\sup_{L\in\bN}\E[N_L^D(E,H_{\argmt})]
    \qquad(\omega\in\Omega_E,E\in\bR)\text,
\end{equation*}
cf.~\cite{Kirsch-Martinelli-1982-JPhysA}.

We are now in position to define the IDS of
$(H_\omega)_{\omega\in\Omega}$ in two steps.
For $\sharp\in\{D,N,M\}$ and
$\omega\in\widetilde\Omega:=\Isect_{E\in\bQ}\Omega_E$, the function
\begin{equation*}
  \widetilde N^\sharp\from\bQ\to\bR\textq,
  \widetilde N^\sharp(E'):=\lim_{L\to\infty}N_L^\sharp(E',H_\omega)
\end{equation*}
is well-defined and non-decreasing.
We extend it to all of~$\Omega$ in such a way that is becomes independent of
$\omega\in\Omega$.
The IDS $N\from\bR\to\bR$ of~$(H_\omega)_{\omega\in\Omega}$
is the right continuous version of~$\widetilde N^\sharp$:
\begin{equation*}
  N(E):=\lim_{\bQ\ni E'\dnto E}\widetilde N^\sharp(E')
    =\inf\bigl\{\widetilde N^\sharp(E')\mid E'\in\bQ\isect\Ioo E\infty\bigr\}
    \text.
\end{equation*}

\begin{remark}[Independence of boundary conditions]
  As indicated by the notation, the IDS~$N$
  is independent of the choice of boundary conditions.
  This can be infered from
  \cite{Kirsch-Martinelli-1982-JPhysA,HundertmarkS-04} as follows.
  See also \cite{HupferLMW-01b,DoiIM-01}.
  Since Neumann and Dirichlet boundary conditions bracket
  Mezincescu boundary conditions, it suffices to show
  \begin{equation*}
    N_L^N(E,H_\omega)-N_L^D(E,H_\omega)\xto{L\to\infty}0
    \qquad\text{for a.\,a.~$\omega$ and all~$E\in\bR$.}
  \end{equation*}
  This is proved in \cite[Theorem~3.3]{Kirsch-Martinelli-1982-JPhysA}
  under the additional assumption that the Laplace transform
  \begin{equation*}
    \E\bigl[\Tr\exp(-t_0(-\Laplace^{N,\cD}+q(\Vper+W_{\argmt})))\bigr]<\infty
  \end{equation*}
  for some $q>1$, $t_0>0$.
  Here, $\Tr$ denotes the trace of operators.
  If we split $\Vper=V_+-V_-$, $V_+,V_-\ge0$, by \cite{HundertmarkS-04},
  the above condition is satisfied as soon as
  $V_++W_\omega\in L_{\mathrm{loc}}^1(\cD)$,
  $V_-$ is relatively form bounded with respect to~$-\Laplace^{N,\cD}$, and
  \begin{equation*}
    \Tr\exp(\Laplace^{N,\cD}-2V_-)<\infty\text.
  \end{equation*}
  But all this follows from $\Vper\in L^\infty(\bR^d)$, \eqref{eq:Wlocp},
  and \cite[Proposition~2.1]{Kirsch-Martinelli-1982-JPhysA}.
\end{remark}

The right continuity of~$N$ implies
\begin{equation*}
  N(E)\ge\inf_{L\in\bN}\E[N_L^\sharp(E,H_{\argmt})]
    \qquad(\sharp\in\{N,M\})
\end{equation*}
for all $E\in\bR$.
Indeed, for all $\ve>0$ we find an $E'\in\bQ\isect\Ioo E\infty$
and an $L\in\bN$ with
\begin{equation*}
  N(E)+2\ve
    \ge\widetilde N^\sharp(E')+\ve
    \ge\E[N_L^\sharp(E',H_{\argmt})]
    \ge\E[N_L^\sharp(E,H_{\argmt})]\text.
\end{equation*}

Actually, for every continuity point $E\in\bR$ of~$N$ we have
\begin{equation}\label{eq:IDSinfL}
  N(E)=\inf_{L\in\bN}\E[N_L^\sharp(E,H_{\argmt})]
    \qquad(\sharp\in\{N,M\})\text.
\end{equation}
Indeed, by continuity of~$N$ in~$E$, there exists for all
$\ve>0$ a $\delta>0$ such that for all $L\in\bN$
and $E',E''\in\bQ$ such that $E-\delta<E'<E''<E$, and $\sharp\in\{N,M\}$,
\begin{align*}
  N(E)-\ve&
    \le N(E')
    \le\widetilde N^\sharp(E'')
    \le\E[N_L^\sharp(E'',H_{\argmt})]
    \le\E[N_L^\sharp(E,H_{\argmt})]\text.
\end{align*}
By an analogous argument we have for continuity points $E\in\bR$ of~$N$
\begin{equation}\label{eq:IDSsupL}
  N(E)=\sup_{L\in\bN}\E[N_L^D(E,H_{\argmt})]\text.
\end{equation}
Note, that in \cite{BourgainK-13} the continuity of~$N$ on all of $\bR$
is proved for $d\in\{1,2,3\}$ and bounded~$\Vper$ and $W_\omega$.
For specific types of~$W_\omega$, Wegner estimates are available,
implying the continuity of the IDS,
see \eg\ \cite{Veselic-08} and references therein.
For breather potentials this is discussed specifically in
\cite{TaeuferV-15,NakicTTV-18b}.

\subsection{Boundary Conditions}\label{sectionboundaryconditions}
We have used above that the eigenvalue counting functions
are equivariant and sub- resp.\ superadditive.
These properties are inherited from the corresponding properties
of the respective operator family
$(H_\omega^{\Lambda,\sharp})_{\omega\in\Omega,\Lambda\in\cF}$.
The latter is \emph{equivariant}, if
$U(x)^{-1}H_\omega^{\Lambda+x,\sharp}U(x)
  =H_{\vartheta(x,\omega)}^{\Lambda,\sharp}$,
for all $x\in\cL$, all $\Lambda\in\cF$ and almost all~$\omega\in\Omega$.
All boundary conditions we consider lead to equivariant operator families.

With Dirichlet boundary conditions, the family
$(H_\omega^{\Lambda,D})_{\omega\in\Omega,\Lambda\in\cF}$
is \emph{subadditive}, meaning that for all disjoint unions
$\Lambda=\Union\Lambda_j\in\cF$ of cubes $\Lambda_j\in\cF$, we have
\begin{equation*}
  H_\omega^{\Lambda,D}
    \le\Directsum\nolimits H_\omega^{\Lambda_j,D}\text,
\end{equation*}
see \cite[p.~270, Proposition~4]{ReedS-78}.
\emph{Superadditivity} is defined with the opposite inequality
and applies for Neumann boundary conditions:
\begin{equation*}
  \Directsum\nolimits H_\omega^{\Lambda_j,N}
    \le H_\omega^{\Lambda,N}\text.
\end{equation*}

Equivariance of~$N^\sharp$ is implied by equivariance of~%
$(H_\omega^{\Lambda,\sharp})_{\omega\in\Omega,\Lambda\in\cF}$.
Superadditivity of $(H_\omega^{\Lambda,\sharp})$
implies subadditivity of~$n^\sharp$, and subadditivity of
$(H_\omega^{\Lambda,\sharp})$ implies superadditivity of~$n^\sharp$.
This is because the smaller operator does not have less eigenvalues
below a given threshold.

In the case~$\Vper=0$, Neumann boundary conditions work well for our purposes.
Otherwise we have to resort to Mezincescu boundary conditions.
Like Neumann boundary conditions they lead to equivariance and superadditivity.
But they additionally preserve the ground state energy of the periodic operator:
$\inf\sigma(\Hper^{\Lambda,M})=\inf\sigma(\Hper)$, $\Lambda\in\cF$.

Following \cite{Mezincescu-87,KirschW-05,KirschW-06,KirschV-10},
we define Mezincescu boundary conditions as Robin boundary conditions
with a specific function
$\rho_\Lambda\in L^\infty(\partial\Lambda,\sigma_\Lambda)$,
where $\sigma_\Lambda$ is the surface measure on the boundary~$\partial\Lambda$ of~$\Lambda$.
This means, $-\Laplace^{\Lambda,M}$
is the operator associated with the sesquilinear form
\begin{equation}\label{eq:Mezincescu-form}
  (\varphi,\psi)\mapsto\Int_\Lambda\cc{\grad\varphi(x)}\grad\psi(x)dx
    +\int_{\partial\Lambda}
      \cc{\varphi(x)}\psi(x)\rho_\Lambda(x)\,\sigma_\Lambda(\dx)
\end{equation}
with the Sobolev space~$H^1(\Lambda)$ as its form domain.
Here and in the following, we use the same name for a function on~$\Lambda$
and for its trace on~$\partial\Lambda$.
The domain of~$-\Laplace^{\Lambda,M}$ turns out to be the set of
$\varphi\in H^2(\Lambda)$ which satisfy
$\rho_\Lambda\varphi+\frac{\partial\varphi}{\partial n}=0$
on~$\partial\Lambda$,
where $\frac\partial{\partial n}$ is the outer normal derivative.
Also, on its domain, $-\Laplace^{\Lambda,M}$
acts as usual as the negative of the sum of the second derivatives.
The details for the case of Neumann boundary conditions $\rho_\Lambda=0$
can be found found in \eg\ \cite[section 10.6.2]{Schmuedgen-12}.
The general case $\rho_\Lambda\in L^\infty(\Lambda)$
can be established analogously.

Note that for all $\rho_\Lambda\in L^\infty(\partial\Lambda)$,
we have
\begin{equation}\label{eq:D_ge_M}
  H_\omega^{\Lambda,D}\ge H_\omega^{\Lambda,M}
\end{equation}
in the form sense, where
$H_\omega^{\Lambda,M}:=-\Laplace^{\Lambda,M}+\Vper+V_\omega$.
In fact, the form which defines the Dirichlet Laplace operator
is the restriction of~\eqref{eq:Mezincescu-form} to $H_0^1(\Lambda)$.

Mezincescu's choice for the function~$\rho_\Lambda$ is constructed as follows.
Note that the restriction $\Hper^{\cD,P}$ of~$\Hper$
to $\cD=\M\Ioo{-\frac12}{\frac12}^d$
with periodic boundary conditions has a positive and normalized ground state
$\Psi_\cD\in H^1(\cD)$, $\norm[2]{\Psi_\cD}=1$.
If we extend $\Psi_\cD$ periodically to~$\bR^d$, we obtain
\begin{equation}\label{eqPsi}
  \Psi\in L^\infty(\bR^d)\isect H_{\mathrm{loc}}^2(\bR^d)\text.
\end{equation}
In fact, by \cite[Theorem~B.3.5]{Simon-82c},
$\Psi$ is continuously differentiable with a H\"older continuous gradient.
Note that in the notation of \cite[(A20)]{Simon-82c},
the regularity $\Vper\in L_u^\infty=L_{\mathrm{loc,unif}}^\infty(\bR^n)$,
guarantees $\Vper\in K_d^{(\alpha)}$ for all $\alpha\in\Ioo01$,
cf.~\cite[Definition~(2) on page 467 and Remark~2 on page 468]{Simon-82c}.
Since~$\Hper$ is elliptic and has regular coefficients, Harnack's inequality,
see \cite[Theorem C.1.3]{Simon-82c}, applies:
\begin{equation}\label{eq:harnack}
  0<\Psi_-:=\min\Psi(\bR^d)\le\Psi_+:=\max\Psi(\bR^d)<\infty\text.
\end{equation}
We now define $\rho_\Lambda:=-\frac1\Psi\frac{\partial\Psi}{\partial n}$
on $\partial\Lambda$.
Equivariance of $(H_\omega^{\Lambda,M})_{\omega\in\Omega,\Lambda\in\cF}$
is clear from construction, and super-additivity
is shown in \cite[Proposition~1]{Mezincescu-87}.

Next, we argue that, for all $\Lambda\in\cF$,
\begin{equation}\label{eq:groundstates}
  E_1(\Hper^{\Lambda,P})
  =E_1(\Hper^{\Lambda,M})
  =\inf\sigma(\Hper)
  \text.
\end{equation}
By construction, $\Psi$~satisfies the eigenvalue equation
$\Hper\Psi=E_1(\Hper^{\cD,P})\Psi$.
Since~$\Psi$ is bounded, \cite[Theorem C.4.1]{Simon-82c} implies
$E_1(\Hper^{\cD,P})\in\sigma(\Hper)$.
On the other hand, $\Psi$ is positive,
so that $E_1(\Hper^{\cD,P})\le\inf\sigma(\Hper)$
by \cite[Theorem C.8.1]{Simon-82c}.
Thus we conclude $E_1(\Hper^{\cD,P})=\inf\sigma(\Hper)$.
Furthermore, for all $\Lambda\in\cF$, the function $\ifu\Lambda\Psi$
is in the domains of~$\Hper^{\Lambda,P}$ and of~$\Hper^{\Lambda,M}$
and an eigenvector with the eigenvalue $E_1(\Hper^{\cD,P})$
for both operators.
Again by positivity and \cite[Theorem C.8.1]{Simon-82c},
$\ifu\Lambda\Psi$ is the ground state of~$\Hper^{\Lambda,P}$
and of~$\Hper^{\Lambda,M}$.
That proves~\eqref{eq:groundstates}.

The following proposition is adapted from \cite[Proposition~4]{Mezincescu-87}.
\genericConstant{Cgap}%
\begin{prop}
  The spectral gap between the two lowest eigenvalues of~$\Hper^{L,M}$
  satisfies, for some $\tilde\Cgap>0$ and all $L\in\bN$,
  \begin{equation*}
    E_2(\Hper^{L,M})-E_1(\Hper^{L,M})
      \ge\tilde\Cgap/(2L+1)^2\text.
  \end{equation*}
\end{prop}
We simplify this estimate to
\begin{equation}\label{e-Mezgap}
  E_2(\Hper^{L,M})-E_1(\Hper^{L,M})
    \ge\Cgap/L^2
\end{equation}
for all $L\ge1$ with $\Cgap:=\frac1{25}\tilde\Cgap$.
\begin{proof}
  For $\sharp\in\{M,N\}$, the set
  $\cC^{L,\sharp}:=\dom(\Laplace^{L,\sharp})\isect L^\infty(\Lambda_L)$,
  is a core of~$\Laplace^{L,\sharp}$.
  The Kato-Rellich theorem, see \eg\ \cite[Theorem~X.12]{ReedS-75},
  guarantees that the core remains the same for~$\Hper^{L,\sharp}$.
  \par
  As above, we denote the ground state of~$\Hper^{L,M}$ by~$\Psi$.
  We abbreviate $E_1:=E_1(\Hper^{L,M})$.
  Using the fact that the ground state~$\Psi$
  is bounded and continuously differentiable, we calculate for all
  $f\in H^2(\Lambda_L)$
  \begin{equation*}
    \partial_n(f\Psi)+\rho_{\Lambda_L}f\Psi
    =(\partial_nf)\Psi+(\partial_n\Psi+\rho_{\Lambda_L}\Psi)f
    =(\partial_nf)\Psi\text.
  \end{equation*}
  With $0<\Psi\in H^2(\Lambda_2)\isect C^1(\Lambda_L)$,
  we conclude that for all $f\in H^2(\Lambda_L)\isect L^\infty(\Lambda_L)$
  \begin{equation*}
    f\Psi\in\cC^{L,M}
    \iff
    f\in\cC^{L,N}
  \end{equation*}
  The eigenvalue equation implies
  \begin{equation}\label{eq:eigenvalueeq}
    (\Vper-E_1)(f\Psi)=f\Laplace^{L,M}\Psi\text.
  \end{equation}
  Hence, we see that for all $f\in\cC^{L,N}$
  \begin{align*}&
    \spr{f\Psi}{(\Hper^{L,M}-E_1)(f\Psi)}\\&
    \stackrel{\text{\eqref{eq:eigenvalueeq}}}
    =\spr{f\Psi}{(-\Laplace^{L,M})(f\Psi)}-\spr{\abs f^2\Psi}{(-\Laplace^{L,M})\Psi}\\&
    \stackrel{\text{\eqref{eq:Mezincescu-form}}}
    =\int_{\Lambda_L}\bigl(\abs{\nabla(f\Psi)}^2
      -\nabla(\abs f^2\Psi)\nabla\Psi\bigr)\dx\\&
    =\int_{\Lambda_L}\Bigl(\abs{(\nabla f)\Psi+f\nabla\Psi}^2
      -\bigl((\nabla\cc f)f\Psi+\cc f(\nabla f)\Psi
      +\abs f^2\nabla\Psi\bigr)\nabla\Psi\Bigr)\dx\\&
    =\int_{\Lambda_L}\abs{\nabla f}^2\Psi^2\dx
    \text.
  \end{align*}
  The boundary terms in the step that uses~\eqref{eq:Mezincescu-form}
  cancel each other.
  \par
  We can now estimate the spectral gap:
  \begin{align*}&
    E_2(\Hper^{L,M})-E_1(\Hper^{L,M})\\&
    =\sup_{\psi\in L^2(\Lambda_L)}
      \inf_{\substack{\varphi\in\cC^{L,M}\\\spr\psi\varphi=0}}
      \frac{\spr\varphi{(\Hper^{L,M}-E_1)\varphi}}{\spr\varphi\varphi}\\&
    \stackrel{\psi:=\Psi^{-1}}
    \ge\inf_{\substack{f\in\cC^{L,N}\\\int_{\Lambda_L}f\ddx=0}}
      \frac{\spr{f\Psi}{(\Hper^{L,M}-E_1)(f\Psi)}}{\spr{f\Psi}{f\Psi}}\\&
    =\inf_{\substack{f\in\cC^{L,N}\\\int_{\Lambda_L}f\ddx=0}}
      \frac{\int_{\Lambda_L}\abs{\nabla f}^2\Psi^2\dx}{\spr{f\Psi}{f\Psi}}\\&
    \ge\frac{\Psi_-^2}{\Psi_+^2}
      \inf_{\substack{f\in\cC^{L,N}\\\int_{\Lambda_L}f\ddx=0}}
      \frac{\int_{\Lambda_L}\abs{\nabla f}^2\dx}{\spr ff}
    \text.
  \end{align*}
  The last infimum is the spectral gap of the Neumann Laplacian on~$\Lambda_L$.
  A straight forward scaling argument finishes the proof.
  We denote $\tilde x:=(2L+1)^{-1}x$
  and $\tilde f(\tilde x):=f((2L+1)\tilde x)$ and see that
  \begin{align*}
    \int_{\Lambda_L}\abs{(\nabla f)(x)}^2\ddx&
    =\int_{\Lambda_{1/2}}\abs{(\nabla f)((2L+1)\tilde x)}^2(2L+1)^d\dd\tilde x\\&
    =\int_{\Lambda_{1/2}}\frac{\abs{(\nabla\tilde f)(\tilde x)}^2}
      {(2L+1)^2}(2L+1)^d\dd\tilde x\text.
  \end{align*}
  The denominator becomes
  \begin{equation*}
    \spr ff
    =\int_{\Lambda_{1/2}}\abs{\tilde f}^2(2L+1)^d\ddx\text,
  \end{equation*}
  and the infimum satisfies
  \begin{equation*}
    \inf_{\substack{f\in\cC^{L,N}\\\int_{\Lambda_L}f\ddx=0}}
      \frac{\int_{\Lambda_L}\abs{\nabla f}^2\dx}{\spr ff}
    =\frac1{(2L+1)^2}
      \inf_{\substack{f\in\cC^{\frac12,N}\\\int_{\Lambda_{1/2}}f\ddx=0}}
      \frac{\int_{\Lambda_{1/2}}\abs{\nabla\tilde f}^2\dx}{\spr{\tilde f}{\tilde f}}
    \text.
  \end{equation*}
  This finishes the proof.
\end{proof}

\section{Lifshitz tails for the breather model}\label{s:breather}
For the breather model, we consider a specific measurable space
$\Omega_0:=\Icc01$ and a specific single site potential
$u\from\Omega_0\times\bR^d\to\bR$.
Choose a measurable set
$A\subseteq\cD=\M\Ioo{-\frac12}{\frac12}^d\subseteq\bR^d$
of positive Lebesgue measure and a coupling strength~$\mu>0$.
The single site potential is defined via
\begin{equation}\label{breatherSSP}
  u(\lambda,x)
  :=\mu\ifu{\lambda A}(x)
  \text,
\end{equation}
see \cref{fig:breather,fig:genbreather} for illustrations.
\begin{figure}
  \begin{tikzpicture}
    [baseline=0mm
    ,mylabel/.style={fill=white,inner xsep=.1mm,inner ysep=.5mm}
    ]
    \foreach \l/\f/\t/\a in {10/{(-1.4,1)}/1/east%
                            ,6.7/{(-1.7,0)}/{2/3}/east%
                            ,3.3/{(-1.5,-1)}/{1/3}/east%
                            } {
      \node [circle,draw,inner sep=\l mm,outer sep=0mm,fill=white] (n) at (0,0) {};
      \node [anchor=\a,at=\f] (t) {$\ifthenelse{\equal{\t}{1}}{A=}{}\supp u_{\t}$};
      \draw[->] (t) -- (n);
      }
    \draw [fill=black] circle (.3mm)
      node at (-1mm,-2mm) {$0$};
  \end{tikzpicture}
  \caption{Support of single site potential $u_\lambda$
    for different values of $\lambda$
    with circular base set~$A$}
  \label{fig:breather}
\end{figure}
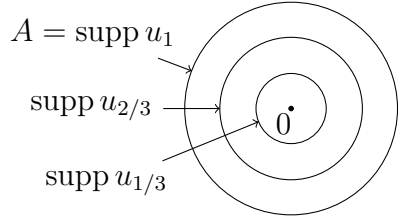%
As in \cref{subsect:IntroIDS},
the random variables $\lambda_k\from\Omega\to\Icc01$, $k\in\cL$,
shall be independent and identically distributed.
We require~$0$ to be in the support of the distribution of~$\lambda_0$:
\begin{equation}\label{supportcondition}
  \forall\ve>0\colon\Prob\{\lambda_0\le\ve\}>0\text,
\end{equation}
and to exclude a trivial case, we assume
\begin{equation}\label{nontrivial}
  \Prob\{\lambda_0=0\}<1\text.
\end{equation}
Note that the distribution of~$\lambda_0$
may but does not have to have an atom at~$0\in\bR$.
Choosing this specific type of single site potential~\eqref{breatherSSP}
in the random potential in~\eqref{e-random-potential} and
the corresponding random \Schroedinger\ operator~$(H_\omega)_{\omega\in\Omega}$
gives rise to the \emph{random breather model}.
\begin{figure}
  \begin{tikzpicture}
    [baseline=0mm
    ,mylabel/.style={fill=white,inner xsep=.1mm,inner ysep=.5mm}
    ]
    \draw [gray] (0,0) -- (5,-2.5);
    \draw [gray] (0,0) -- (4,2);
    \draw [fill=black] circle (.3mm)
      node [label=above:$0$] {}
      node [mylabel] at (-0.1,-0.4) {$\supp u_0$};
    \foreach \l in {1,{2/3},{1/3}} {
      \draw [fill=gray] (5*\l,-2.5*\l)
        .. controls (8*\l,-2.5*\l) and (6*\l,2*\l)  .. (4*\l,2*\l)
        .. controls (6*\l,1*\l)    and (4*\l,-2*\l) .. (5*\l,-2.5*\l);
      \draw node [mylabel] at (5*\l-0.1,-2.5*\l-0.4) {$\supp u_{\l}$};
      }
    \draw node at (5.6,-1.1) {$A$};
  \end{tikzpicture}
  \caption{Support of single site potential $u_\lambda$
    for different values of $\lambda$
    with arbitrary base set~$A$}
  \label{fig:genbreather}
\end{figure}
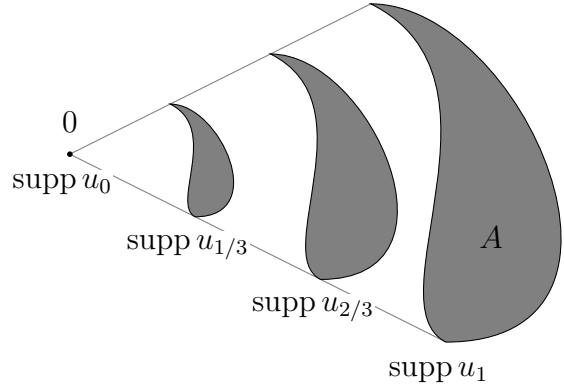%

The central result of this chapter is the following.
\genericConstant{Cfront}%
\genericConstant{Cexp}%
\begin{thm}\label{thmbreather}
  The IDS~$N$ of the random breather model with single site potentials given by
  \eqref{breatherSSP} satisfies a Lifshitz bound, \ie\ %
  $\exists\Cfront,\Cexp>0,E'>E_0\colon\forall E\in\Ioc{E_0}{E'}\colon$
  \begin{equation}\label{e-LifExp}
    N(E)\le\Cfront
    \exp\bigl(-\Cexp(E-E_0)^{-d/2}\bigr)
    \text.
  \end{equation}
\end{thm}

\subsection{Identification of $E_0$}
Before we prove \cref{thmbreather} in \cref{subsectionProofBreather},
we show that the support condition~\eqref{supportcondition}
guarantees that there actually is spectrum at~$E_0$.

\genericConstant{perappEF}%
\begin{lemma}\label{E0inf}
  For the random breather model with \eqref{supportcondition}, we have
  \begin{equation}\label{eq:E0inf}
    E_0=\inf\sigma(H_\omega)
      \qqtext{for\/ $\Prob$-a.\,a.~$\omega\in\Omega$.}
  \end{equation}
\end{lemma}
\begin{proof}
  We have $W_\omega\ge0$ for almost all $\omega\in\Omega$.
  Hence by the min-max-principle $\inf\sigma(H_\omega)\ge E_0$.
  \par
  For the reverse inequality,
  we apply the min-max-principle with a test function to~$H_\omega$
  for a.\,a.\ $\omega\in\Omega$.
  We use a smooth cut-off function $\molly{}\in C^\infty(\bR^d,\Icc01)$
  with $\ifu\cD\le\molly{}\le\ifu{\Lambda_{1/2}}=\ifu{\M\Ioo{-1}1^d}$.
  We scale by $L\in\bN$ and shift $x\in\cL$:
  \begin{equation}\label{eq:molly}
    \molly{x,L}(y):=\molly{}\bigl((y-x)/L\bigr)
    \qquad(y\in\bR^d)\text,
  \end{equation}
  so that $\supp\molly{x,L}\subseteq\Lambda_L(x)$.
  Our test function is
  \begin{equation}\label{eq:psixL}
    \psi_{x,L}:=\Psi\cdot\molly{x,L}/\norm[2]{\Psi\cdot\molly{x,L}}
    \in\dom(\Laplace)\text,
  \end{equation}
  with $\Psi$ from \eqref{eqPsi},
  for suitable $x\in\cL$ and $L\in\bN$ to be chosen later.
  We have to bound
  \begin{equation*}
    \norm[2]{(H_\omega-E_0)\psi_{x,L}}
    \le\norm[2]{{(\Hper-E_0)\psi_{x,L}}}+\norm[2]{W_\omega\psi_{x,L}}
    \text.
  \end{equation*}
  \Cref{approxEF,BC} provide a constant $\perappEF>0$
  such that, for $\Prob$-almost all $\omega\in\Omega$
  and all $L\in\bN$, there exists an $x\in\cL$ such that
  \begin{align*}
    \norm[2]{(H_\omega-E_0)\psi_{x,L}}&
      \le(\perappEF+1)/L\text.
  \end{align*}
  This suffices to conclude
  $\inf\sigma(H_\omega)
   =\inf\limits_{\psi\in\dom(H_\omega),\norm[2]\psi=1}\spr\psi{H_\omega\psi}
   \le E_0$.
\end{proof}

The following lemma provides the non--random estimate used in \cref{E0inf}.
\begin{lemma}\label{approxEF}
  There exists a constant~$\perappEF>0$ such that,
  for all $L\in\bN$ and $x\in\cL$
  \begin{equation*}
    \norm[2]{(\Hper-E_0)\psi_{x,L}}
      \le\perappEF/L\text.
  \end{equation*}
\end{lemma}
\begin{proof}
  We use $\Hper\Psi=E_0\Psi$ for
  \begin{align*}
    \norm[2]{(\Hper-E_0)(\Psi\molly{x,L})}&
      =\norm[2]{2\nabla\Psi\cdot\nabla\molly{x,L}+\Psi\Laplace\molly{x,L}}\\&
      \le2\norm[2]{\ifu{\Lambda_L(x)}\nabla\Psi}\norm[\infty]{\nabla\molly{x,L}}
        +\norm[\infty]{\Psi}\norm[2]{\Laplace\molly{x,L}}\text.
  \end{align*}
  A short calculation shows the $x$-independent bound
  \begin{equation*}
    \norm[\infty]{\nabla\molly{x,L}}
      =L^{-1}\norm[\infty]{\nabla\molly{}}
    \qtextq{and}
    \norm[2]{\Laplace\molly{x,L}}
      =L^{-2+(d/2)}\norm[2]{\Laplace\molly{}}
    \text.
  \end{equation*}
  Furthermore, we have
  \begin{equation*}
    \norm[2]{\ifu{\Lambda_L(x)}\nabla\Psi}
      =(2L+1)^{d/2}\norm[2]{\ifu\cD\nabla\Psi}
  \end{equation*}
  and
  \begin{equation}\label{eq:lowerPsimolly}
    \norm[2]{\Psi\molly{x,L}}
      \ge
      L^{d/2}\norm[2]{\Psi\ifu\cD}
      =L^{d/2}
    \text.
  \end{equation}
  We combine this to get
  \begin{equation*}
    \norm[2]{(\Hper-E_0)\psi_{x,L}}
      \le\frac
          {2(2+L^{-1})^{d/2}
            \norm[2]{\ifu\cD\nabla\Psi}\norm[\infty]{\nabla\molly{}}
          +\Psi_+\norm[2]{\Laplace\molly{}}/L}
          L
  \end{equation*}
  and choose $\perappEF:=2\cdot3^{d/2}\norm[2]{\ifu\cD\nabla\Psi}
    \norm[\infty]{\nabla\molly{}}+\Psi_+\norm[2]{\Laplace\molly{}}$.
\end{proof}

\Cref{BC} deals with the random part in the estimate from \cref{E0inf}.
\begin{lemma}\label{BC}
  There exists a set $\Omega_{BC}\in\cA$
  of full probability $\Prob(\Omega_{BC})=1$,
  such that for all $\omega\in\Omega_{BC}$, $\ve>0$ and $L\in\bN$,
  there exists $x\in\cL$ satisfying
  \begin{equation*}
    \norm[2]{W_\omega\psi_{x,L}}\le\ve\text.
  \end{equation*}
\end{lemma}
\begin{proof}
  Let $I_L:=\cL\isect\Lambda_L$ and $I_L(x):=x+I_L$.
  For all $\alpha>0$ and $L\in\bN$,
  the Borel--Cantelli lemma provides us with a set $\Omega_{\alpha,L}\in\cA$
  of full measure, such that for all $\omega\in\Omega_{\alpha,L}$,
  there exists $x=x_{\alpha,L,\omega}\in\cL$ such that
  \begin{equation}\label{eq:BC}
    \sup_{k\in I_L(x)}\lambda_k(\omega)\le\alpha\text.
  \end{equation}
  We let $\Omega_{BC}
    :=\Isect_{\alpha\in\bQ,\alpha>0}\Isect_{L\in\bN}\Omega_{\alpha,L}\in\cA$
  and note that $\Prob(\Omega_{BC})=1$,
  and that~\eqref{eq:BC} holds for all $\alpha>0$, $L\in\bN$,
  and $\omega\in\Omega_{BC}$.
  \par
  Using~\eqref{eq:BC}, we calculate
  \begin{equation*}
    \norm[2]{W_\omega\ifu{\Lambda_L(x)}}^2
    \le\mu^2\alpha^d\vol A\setsize{I_L}
  \end{equation*}
  and, together with~\eqref{eq:lowerPsimolly},
  \begin{align*}
    \norm[2]{W_\omega\psi_{x,L}}&
      \le\frac{\norm[\infty]\Psi}{\norm[2]{\Psi\molly{x,L}}}
        \norm[2]{W_\omega\ifu{\Lambda_L(x)}}\\&
      \le\frac{\Psi_+\mu\sqrt{\alpha^d\vol A(2L+1)^d}}{L^{d/2}}
      \le\Psi_+\mu\sqrt{\vol A}(3\alpha)^{d/2}
      \text.
  \end{align*}
  Choosing
  $\alpha:=\frac13\bigl(\ve/(\Psi_+\mu\sqrt{\vol A})\bigr)^{2/d}>0$
  finishes the proof.
\end{proof}

\subsection{Temple vs.\ Thirring}

In \cite{KirschV-10} a Lifshitz bound was proven
for a restricted class of breather potentials.
The method of proof used there is not applicable
to the potential~\eqref{breatherSSP}, at least not directly,
because of the use of Temple's inequality~\cite{Temple-28}.
Let us explain why:
Provided that~$\psi$ is a normalized vector in the domain
of a lower bounded self-adjoint operator~$H$ with
$E_1(H)=\inf\sigma(H)$ being a simple eigenvalue and
$E_2(H)=\inf(\sigma(H)\setminus\{E_1(H)\})$
being an eigenvalue, and $\nu\in\bR$ a number such that $\inf\sigma(H)\le\spr\psi{H\psi}<\nu\le E_2(H)$,
then Temple's inequality states that the lowest eigenvalue~$E_1(H)$
of~$H$ is lower bounded by
\[
E_1(H)\ge
\spr \psi {H \psi} - \frac{\norm{ H \psi }^2 -(\spr \psi {H \psi})^2}{\nu  -\spr \psi {H \psi}}\text.
\]
Here $\spr \psi {H \psi}$ can be considered as the first order approximation to $E_1(H)$.
Indeed, it is a true \emph{upper bound}. To obtain a lower bound one has to subtract a normalization of the second order correction
$\norm{ H \psi }^2 -(\spr \psi {H \psi})^2$.
This expression is non-negative, since it is a variance.
In our application, considering eigenvalues close to zero,
$\spr \psi {H \psi}$ is small, consequently is $(\spr \psi {H \psi})^2$
quadratically small and thus negligible.
If the single site potential $u(\lambda,x)= \ifu{\lambda A}(x)$
is an indicator function of~$\lambda A$, we have
\begin{equation*}
  \spr\psi{u(\lambda,\cdot)^2\psi}
  =\spr\psi{u(\lambda,\cdot)\psi}\text.
\end{equation*}
If the translates $k+\lambda A, k \in \cL$,
do not overlap for any allowed value of~$\lambda$, \eg\ if is~$A$ is small,
then we have for the resulting breather potential~$W_\omega$
in~\eqref{e-random-potential} again
\begin{equation*}
  \spr\psi{W_\omega^2\psi}
  =\spr\psi{W_\omega\psi}\text.
\end{equation*}
The natural choice of a test function~$\psi$
is the ground state of $\Hper^{L,M}$ with eigenvalue $E_1(\Hper^{L,M})=0$.
Then we have
\begin{equation*}
  \norm{H_\omega^{L,M}\psi}
  =\spr\psi{H_\omega^{L,M}\psi}\text.
\end{equation*}
Hence the second moment is equal to the first one
and cannot be considered as small correction.
Note that the difference $\nu-\spr\psi{H_\omega^{L,M}\psi}$
is bounded by the gap between the first two eigenvalues of~$H_\omega^{L,M}$,
typically of order~$L^{-2}$.
Thus, dividing by this number actually makes the correction term even larger.

It turns out that Thirring's inequality \cite[\nopp3.5.32]{Thirring-94}
is better adapted to the model under consideration.
It was introduced in the context of random operators in \cite{KirschM-83a},
but then abandoned in later literature in favor of Temple's inequality.
For the readers convenience we reprove Thirring's inequality here.

\begin{lemma}[Thirring]\label{lemma:Thirring}
  Let~$V$ be an invertible, positive operator on the Hilbert space~$\cH$,
  $P\from\cH\to P(\cH)\subseteq\cH$~an orthogonal projection and suppose,
  that the operator $PV^{-1}P^*\in\cB(P\cH)$ is invertible.
  Then
  \begin{equation*}
    P^*(PV^{-1}P^*)^{-1}P
      \le V\text.
  \end{equation*}
  In consequence, if~$H\from D(H)\to\cH$ is a self-adjoint operator on~$\cH$,
  bounded from below, and
  \begin{equation*}
    E_1(H)\le E_2(H)\le E_3(H)\le\dotsb
  \end{equation*}
  is the sequence of eigenvalues of~$H$ below $\inf\sigma_\mathrm{ess}(H)$,
  counted with multiplicity, then for all $n\in\bN$
  \begin{equation*}
    E_n\bigl(H+P^*(PV^{-1}P^*)^{-1}P\bigr)\le E_n(H+V)\text.
  \end{equation*}
\end{lemma}
\begin{proof}
  Let $Q:=V^{-\frac12}P^*(PV^{-1}P^*)^{-1}PV^{-\frac12}$.
  Note that because of
  \begin{equation*}
    Q^2
      =V^{-\frac12}P^*(PV^{-1}P^*)^{-1}PV^{-1}P^*(PV^{-1}P^*)^{-1}PV^{-\frac12}
      =Q=Q^*\text,
  \end{equation*}
  $Q$~is itself an orthogonal projection.  Therefore $Q\le\Id[\cH]$,
  \ie\ $\spr\psi{Q\psi}\le\spr\psi\psi$ for all $\psi\in\cH$.
  This directly implies
  \begin{equation*}
    P^*(PV^{-1}P^*)^{-1}P=V^{\frac12}QV^{\frac12}\le V\text.
  \end{equation*}
  By the min-max-principle, see \eg\ \cite[\nopp3.5.21]{Thirring-94},
  for all $n\in\bN$
  \begin{align*}&
    E_n\bigl(H+P^*(PV^{-1}P^*)^{-1}P\bigr)\\&\quad
      =\adjustlimits\inf_{\substack{\cH'\subseteq D(H)\\\dim(\cH')=n}}
        \sup_{\substack{\psi\in\cH'\\\norm\psi=1}}
        \spr\psi{(H+P^*(PV^{-1}P^*)^{-1}P)\psi}\\&\quad
      \le\adjustlimits\inf_{\substack{\cH'\subseteq D(H)\\\dim(\cH')=n}}
        \sup_{\substack{\psi\in\cH'\\\norm\psi=1}}
        \spr\psi{(H+V)\psi}
      =E_n(H+V)
    \text.\qedhere
  \end{align*}
\end{proof}

We actually need only the following special case of \cref{lemma:Thirring}.
\begin{cor}\label{corollary:Thirring}
  Let $\psi\in\cH$ be a normalized ground state of~$H$,
  \ie\ $\norm\psi=1$ and $H\psi=E_1(H)\psi$,
  and $P=P_{\psi}$ the orthogonal projection onto
  $\operatorname{span}\{\psi\}\subseteq\cH$.
  Then
  \begin{equation*}
    \min\{E_1(H)+\spr\psi{V^{-1}\psi}^{-1},E_2(H)\}
      \le E_1(H+V)\text.
  \end{equation*}
\end{cor}
\begin{proof}
  Since $P\cH$ is one dimensional,
  \begin{equation*}
    PV^{-1}P^*
      =\spr{P\psi}{PV^{-1}P^*P\psi}
      =\spr{P^*P\psi}{V^{-1}P^*P\psi}
      =\spr\psi{V^{-1}\psi}\text,
  \end{equation*}
  where the scalar on the right hand side
  is interpreted as multiplication operator on $P\cH$.
  We use the min-max-principle to show
  \begin{equation*}
    E_1(H+P^*(PV^{-1}P^*)^{-1}P)
      \ge\min\{E_1(H)+\spr\psi{V^{-1}\psi}^{-1},E_2(H)\}\text.
  \end{equation*}
  Then \cref{corollary:Thirring} follows from \cref{lemma:Thirring}.

  In order to apply the min-max-principle,
  we decompose the arbitrary vector~$\varphi\in\cH$
  of unit length~$\norm\varphi=1$ into $\varphi=\alpha\psi+\psi^\bot$
  with $\alpha\in\bC$, $\abs\alpha\le1$, and~$\psi^\bot$ orthogonal to~$\psi$.
  Self-adjointness of~$H$ gives us $\spr\psi{H\psi^\bot}=0$.
  In addition we know
  $\spr{\psi^\bot}{H\psi^\bot}
    \ge E_2(H)\norm{\psi^\bot}^2
    =E_2(H)(1-\abs\alpha^2)$.
  We now see that
  \begin{align*}&
    \spr\varphi{(H+P^*(PV^{-1}P^*)^{-1}P)\varphi}\\&\quad
      =\abs\alpha^2\spr\psi{H\psi}+\abs\alpha^2\spr\psi{V^{-1}\psi}^{-1}
        +\spr{\psi^\bot}{H\psi^\bot}\\&\quad
      \ge\abs\alpha^2\bigl(E_1(H)+\spr\psi{V^{-1}\psi}^{-1}\bigr)
        +(1-\abs\alpha^2)E_2(H)
    \text.
  \end{align*}
  The last expression is affine linear in $\abs\alpha^2\in\Icc01$,
  so the minimum is realized for $\abs\alpha^2\in\{0,1\}$.
\end{proof}

\subsection{Proof of Lifshitz asymptotics for the breather model}
\label{subsectionProofBreather}
\genericConstant{Cld}%
\begin{proof}[Proof of \cref{thmbreather}]
  Recall the fundamental domain $\cD=\M\Ioo{-\frac12}{\frac12}^d$
  of the lattice~$\cL=\M\bZ^d$
  and $\Lambda_L:=\M\Ioo[\big]{-(L+\frac12)}{L+\frac12}^d$.
  We define a relevant index set $I_L:=\cL\isect\Lambda_L$ for $L\in\bN$
  and note that~$\Lambda_L$
  is the interior of the closure of the open set
  $I_L+\cD=\Union_{k\in I_L}(k+\cD)$.
  We therefore have $\vol{\Lambda_L}=\setsize{I_L}\cdot\vol\cD$,
  where $\vol{\argmt}$ denotes Lebesgue measure and
  $\setsize{}$ the counting measure.
  Since we will only use Mezincescu boundary conditions in this proof,
  we drop~$M$ from the notation for finite volume restrictions of operators:
  $H^L:=H^{L,M}=H^{\Lambda_L,M}$.

  By adding a constant to the periodic potential~$\Vper$,
  we can without loss of generality simplify our notation and assume
  \begin{equation}\label{E0=0}
    E_0
      =0\text.
  \end{equation}
  We also need to show~\eqref{e-LifExp} only for points~$E$
  of continuity of $N$,
  since~$N$ is monotone and the right hand side of~\eqref{e-LifExp}
  is continuous in~$E$.

  We single out the following properties of the single site potential:
  \begin{equation}\label{eq:OBdA}
    u_{\lambda_k}=\mu\ifu{\supp u_{\lambda_k}}\le\mu\ifu{\cDc}
    \qtextq{$\Prob\times\lambda^d$-a.\,e., and}
    \E\bigl[\vol{\supp u_{\lambda_k}}\bigr]>0
    \text,
  \end{equation}
  where the first conditions holds almost everywhere
  with respect to the product measure on $\Omega\times\bR^d$
  of~$\Prob$ and Lebesgue measure~$\lambda^d$.
  These are the properties of~$u$ we will actually use in this proof.
  This abstraction will be useful in \cref{thmabstract},
  where we can recycle large parts from here.
  Note that we do not require precise information about
  $\supp u_\lambda=\{x\in\bR^d\mid u_\lambda(x)=\mu\}$,
  its suffices to know that~$u$ only assumes the two values~$0$ and~$\mu$,
  is supported in~$\cDc$,
  and that the expected Lebesgue measure of the support of~$u$ is positive.

  For non--negative single site potentials it is well known that,
  using \eqref{eq:IDSinfL},
  \begin{equation}\label{eq:NE}
    \begin{aligned}
      N(E)&
        \le\E[N_L(E,H_{\argmt})]
        =\E\bigl[\ifu{\{E_1(H_{\argmt}^L)\le E\}}N_L(E,H_{\argmt})\bigr]\\&
        \le\E\bigl[\ifu{\{E_1(H_{\argmt}^L)\le E\}}N_L(E,\Hper)\bigr]\\&
        \le N_L(E,\Hper)\Prob\{\omega\colon E_1(H_\omega^L)\le E\}
    \end{aligned}
  \end{equation}
  for $L\in\bN$ and points $E\in\bR$ of continuity of~$N$.
  Since $N_L(E,\Hper)$ is uniformly bounded in $L\in\bN$
  and $E\le1$, see \cite[Prop.~3.1]{Kirsch-Martinelli-1982-JPhysA},
  it is sufficient to derive an exponential bound on the probability
  that the first eigenvalue~$E_1(H_\omega^L)$ of~$H_\omega^L$
  does not exceed~$E$ for a suitably chosen~$L=L_E$.

  In order to apply \cref{corollary:Thirring} of Thirring's inequality,
  we need the potential to be strictly positive.
  Therefore we split the potential in the following way
  \begin{equation*}
    H_0^L:=-\Delta^L+\Vper-\gamma_L\quad\text{and}\quad
    V_\omega:=W_\omega+\gamma_L
  \end{equation*}
  with $\gamma_L:=\Cgap/(2L^2)$ and $\Cgap$ from~\eqref{e-Mezgap}.
  This shift by $\gamma_L$ scales like the gap between the first
  and the second eigenvalue of~$(-\Laplace+\Vper)^L$, cf.~\eqref{e-Mezgap}.

  Due to the normalization~\eqref{E0=0},
  the ground state energy of~$H_0^L$ is $E_1(H_0^L)=-\gamma_L$.
  Furthermore, by \eqref{e-Mezgap},
  \begin{equation}\label{eq:E2}
    E_2(H_0^L)
      \ge\frac{\Cgap}{L^2}-\gamma_L
      =\gamma_L\qquad(L\in\bN)\text.
  \end{equation}

  Recall from~\eqref{eqPsi}
  the $\cL$-periodic function~$\Psi\from\bR^d\to\Ioo0\infty$.
  As pointed out in \cref{sectionboundaryconditions},
  the ground state of~$H_0^L$ is given by
  $\Psi_L:=\ifu{\Lambda_L}\Psi/\sqrt{\setsize{I_L}}$,
  now properly normalized.
  We define for all $k\in\cL$ and $L\in\bN$
  the random variables $X_k,S_L\from\Omega\to\Icc01$ via
  \begin{equation}\label{XS}
    X_k(\omega):=\int_{\supp u_{\lambda_k(\omega)}}\abs{\Psi(x)}^2\,\dx
    \qtextq{and}
    S_L:=\frac1{\setsize{I_L}}\sum_{k\in I_L}X_k
    \text.
  \end{equation}
  The values of~$X_k$ and~$S_L$ are in~$\Icc01$
  because of $\norm[2]{\Psi\ifu{\cDc}}=1$.
  From $\Psi(x)\ge\Psi_->0$ and \eqref{eq:OBdA} we have $\E[X_0]>0$.
  We define the crucial length~$\hat L_E$:
  \begin{equation*}
    \hat L_E:=\floor*{\sqrt{\Cgap/(2E)}}
  \end{equation*}
  and see, using~\eqref{eq:E2}, that for all $L\in\bN$, $L\le\hat L_E$,
  \begin{equation}\label{eq:gamma-sandwich}
    E\le\frac{\Cgap}{2\hat L_E^2}
      =\gamma_{\hat L_E}
      \le\gamma_L
      \le E_2(H_0^L)\text.
  \end{equation}
  On the event $\{\omega\colon E_1(H_\omega^{\hat L})\le E\}$,
  \cref{lem:XSV} below, \eqref{eq:gamma-sandwich},
  and Thirring's inequality~\cref{corollary:Thirring}
  imply for all $L\in\bN$, $L_0\le L\le\hat L_E$,
  \begin{equation*}
    \frac{\gamma_L}2S_L(\omega)
      \le E_1(H_0^L)+\spr{\Psi_L}{V_\omega^{-1}\Psi_L}^{-1}
      \le E_1(H_\omega^L)
      \le E
      \le E_2(H_0^L)\text.
  \end{equation*}
  Let $L_E:=\floor[\big]{\sqrt{\Cgap\E[X_0]/(8E)}}$.
  We use $X_0\le1$ to check that for~$E$ small enough, $L_0\le L_E\le\hat L_E$.
  Hence, since $\E[X_0]=\E[S_L]$ is constant in~$L$,
  \begin{align*}
    \Prob\{\omega\colon E_1(H_\omega^{L_E})\le E\}&
      \le\Prob\bigl\{\tfrac{\gamma_{L_E}}2S_{L_E}\le E\bigr\}
      \le\Prob\{S_{L_E}\le\tfrac12\E[S_{L_E}]\}\text.
  \end{align*}

  Finally, observe that the random variables~$X_k$, $k\in\cL$,
  are independent.
  Chernov's bound, see \cref{LDP},
  estimates the last probability by $\e^{-\Cld(2L_E+1)^d}$
  with some positive constant~$\Cld$,
  since $\setsize{I_{L_E}}=(2L_E+1)^d$.
  Continuing the estimate~\eqref{eq:NE},
  we see from the definition of~$L_E$, that
  \begin{align*}
    N(E)&
    \le\Cfront\exp\bigl(-\Cld(2L_E+1)^d\bigr)\\&
    \le\Cfront\exp\bigl(-\Cld\bigl(2\sqrt{\Cgap\E[X_0]/(8E)}-1\bigr)^d\bigr)\\&
    =\Cfront\exp\bigl(-\Cld\bigl(2\sqrt{\Cgap\E[X_0]/8}-\sqrt E\bigr)^d
      E^{-d/2}\bigr)\\&
    \le\Cfront\exp\bigl(-\Cexp E^{-d/2}\bigr)
  \end{align*}
  for~$E\le\Cgap\E[X_0]/8$,
  where $\Cexp:=\Cld\bigl(\sqrt{\Cgap\E[X_0]/8}\bigr)^d$.
  We reached~\eqref{e-LifExp} and finished the proof of \cref{thmbreather}.
\end{proof}

\begin{lemma}\label{lem:XSV}
  For all $L\in\bN$,
  $L\ge L_0:=\ceil*{\sqrt{\frac{\Cgap}{2\mu}}}$,
  \begin{equation*}
    \frac{\gamma_L}2S_L(\omega)\le
      E_1(H_0^L)+\spr{\Psi_L}{V_\omega^{-1}\Psi_L}^{-1}\text.
  \end{equation*}
\end{lemma}
\begin{proof}[Proof of \cref{lem:XSV}]\label{proofXSV}
  As $V_\omega$ does not vanish,
  $V_\omega^{-1}$ is well-defined as a multiplication operator.
  By construction, we have
  \begin{align*}
    \spr{\Psi_L}{V_\omega^{-1}\Psi_L}&
      =\int\frac{\abs{\Psi_L(x)}^2}{V_\omega(x)}\,\dx
      =\frac1{\setsize{I_L}}
        \int_{\Lambda_L}\frac{\abs{\Psi(x)}^2}{V_\omega(x)}\,\dx\\&
      =\frac1{\setsize{I_L}}\sum_{k\in I_L}
        \int_{\cDc+k}\frac{\abs{\Psi(x)}^2}{V_\omega(x)}\,\dx\text.
  \end{align*}
  Using $\norm[2]{\Psi\ifu{\cDc}}^2=1$, we rewrite the summand,
  introducing~$X_k$:
  \begin{align*}
    \int_{\cDc+k}\frac{\abs{\Psi(x)}^2}{V_\omega(x)}\,\dx&
    =\int_{\cDc}\frac{\abs{\Psi(x)}^2}{u_{\lambda_k(\omega)}(x)+\gamma_L}\,\dx\\&
    =\frac{X_k(\omega)}{\mu+\gamma_L}+\frac{1-X_k(\omega)}{\gamma_L}
    =\frac1{\gamma_L}\frac{\mu+\gamma_L-\mu X_k(\omega)}{\mu+\gamma_L}\text.
  \end{align*}
  Averaging over~$k\in I_L$ gives
  \begin{align*}
    \spr{\Psi_L}{V_\omega^{-1}\Psi_L}&
      =\frac{\mu+\gamma_L-\mu S_L(\omega)}{(\mu+\gamma_L)\gamma_L}\text.
  \end{align*}
  The inequality $L\ge L_0$ implies $\gamma_L\le\mu$.
  Hence, using $S_L\ge0$, too,
  \begin{align*}
    E_1(H_0^L)+\spr{\Psi_L}{V_\omega^{-1}\Psi_L}^{-1}&
      =\gamma_L
        \frac{\mu S_L(\omega)}
             {\mu+\gamma_L-\mu S_L(\omega)}
      \ge\frac{\gamma_L}2S_L(\omega)\text.
      \qedhere
  \end{align*}
\end{proof}

\begin{lemma}\label{LDP}
  Given a sequence~$X_k$, $k\in\bN$,
  of non--negative i.\,i.\,d.\ random variables with
  $0<\E[X_1]<\infty$, let $S_n:=\frac1n\sum_{k=1}^nX_k$.
  Then there exists $\Cld>0$, determined by the law of~$X_1$, with
  \begin{equation*}
    \Prob\{S_n\le\E[S_n]/2\}
      \le\e^{-\Cld n}\text.
  \end{equation*}
\end{lemma}
\begin{proof}
  Observe that for all non--negative numbers~$t\ge0$ by independence
  \begin{align*}
    \Prob\bigl\{S_n\le\tfrac12\E[S_n]\bigr\}&
      =\E\bigl[\ifu{\{\exp(nt(\E[S_n]-2S_n))\ge1\}}\bigr]\\&
      \le\E\bigl[\exp\bigl(nt(\E[S_n]-2S_n)\bigr)\bigr]\\&
      =\prod_{k=1}^n\E\bigl[\exp\bigl(t(\E[X_k]-2X_k)\bigr)\bigr]
      \text.
  \end{align*}
  The identical distribution of the random variables~$X_k$ implies
  \begin{align*}
    \Prob\bigl\{S_n\le\tfrac12\E[S_n]\bigr\}&
      \le\bigl(\E\bigl[\exp\bigl(t(\E[X_1]-2X_1)\bigr)\bigr]\bigr)^n
      =\exp\bigl(n\log M(t)\bigr)\text,
  \end{align*}
  employing the moment generating function
  \begin{equation*}
    M(t):=\E\bigl[\exp\bigl(t(\E[X_0]-2X_0)\bigr)\bigr]
    \quad(t\in\bR)
  \end{equation*}
  of the random variable $\E[X_0]-2X_0$.
  \par
  Note, that $M(0)=1$ and $M'(0)=\E\bigl[\E[X_0]-2X_0\bigr]=-\E[X_0]<0$.
  Therefore we find $s>0$ with $M(s)<1$,
  which proves the claim with $\Cld:=\abs{\log M(s)}>0$.
\end{proof}

\section{Reduction of the general case to the breather model}\label{s:reduction}

In the present section we reduce a far more general situation
to the setting~\eqref{eq:OBdA}.
This shows Lifshitz tails for a broad family of random potentials.
In the generalization we keep the independence of the single site potentials,
but we do not require them to be identically distributed any more.
Instead, we use a less restrictive way to express the necessary regularity.

\begin{definition}\label{def:non-degenerate}
  Let $\lambda_k\from\Omega\to\Omega_0$, $k\in\cL$,
  be a family of random variables on the probability space
  $(\Omega,\cA,\Prob)$, $u\from\Omega_0\times\bR^d\to\Ico0\infty$
  a function and $\mu>0$.
  The function~$u$ is called \emph{$\mu$-non--degenerate},
  if it is jointly measurable and
  \begin{equation}\label{eq:non-degenerate}
    \inf_{k\in\cL}\Prob\bigl\{
      \vol{\{x\in\cDc\mid u_{\lambda_k}(x)\ge\mu\}}\ge\mu
    \bigr\}
      \ge\mu\text.
  \end{equation}
  We say that~$u$ is \emph{non--degenerate}, if there exists a $\mu>0$
  such that $u$ is $\mu$-non--degenerate.
\end{definition}
\Cref{Enon-degenerate} gives a simple criterion for non--degeneracy.
In the stationary case it corresponds to
the assumption in Theorem 4 of \cite{KirschM-83a}.
The criterion allows to see easily that the single site potential
of the breather model~\eqref{breatherSSP} is non--degenerate.

\begin{lemma}\label{Enon-degenerate}
  The function~$u$ is non--degenerate if and only if
  \begin{equation*}
    \inf_{k\in\cL}\E\Int_{\cDc}\min\{u_{\lambda_k}(x),1\}dx>0\text.
  \end{equation*}
Furthermore, under the additional assumption that $u_{\lambda_k}$, $k\in\cL$,
are identically distributed, non--degeneracy is equivalent to
\begin{equation*}
  \Prob\bigl\{\omega\in\Omega\colon\ifu{\cDc}u_{\lambda_k}\not\equiv0\bigr\}
  >0\text.
\end{equation*}
\end{lemma}
\begin{proof}
  Assume that there exists $\mu\in\Ioc01$ with
  \begin{equation*}
    \inf_{k\in\cL}\Prob\bigl\{
      \vol{\{x\in\cDc\mid u_{\lambda_k}(x)\ge\mu\}}\ge\mu
    \bigr\}
      \ge\mu\text.
  \end{equation*}
  Since $\mu\le1$, for all $k\in\cL$, we have using \v Ceby\v sev's inequality
  \begin{align*}
    \E\Int_{\cDc}\min\{u_{\lambda_k}(x),1\}dx&
    \ge\E\Int_{\cDc}\mu\ifu{\{x\in\cDc\mid u_{\lambda_k}(x)\ge\mu\}}dx\\&
    =\mu\E\vol{\{x\in\cDc\mid u_{\lambda_k}(x)\ge\mu\}}\\&
    \ge\mu^2\Prob\bigl\{\vol{\{x\in\cDc\mid
      u_{\lambda_k}(x)\ge\mu\}}\ge\mu\bigr\}\\&
    \ge\mu^3
    >0\text.
  \end{align*}
  For the other direction, assume that for all $\mu\in\Ioc01$,
  \begin{equation*}
    \inf_{k\in\cL}\Prob\bigl\{
      \vol{\{x\in\cDc\mid u_{\lambda_k}(x)\ge\mu\}}\ge\mu
    \bigr\}
      <\mu\text,
  \end{equation*}
  Then there exists $k\in\cL$ such that
  $\Prob\bigl\{
    \vol{\{x\in\cDc\mid u_{\lambda_k}(x)\ge\mu\}}\ge\mu\bigr\}<\mu$.
  Then, we deduce
  \begin{align*}
    \E\Int_{\cDc}\min\{u_{\lambda_k}(x),1\}dx&
    \le\E\bigl[\mu\vol\cDc
      +\vol{\{x\in\cDc\mid u_{\lambda_k}(x)\ge\mu\}}\bigr]\\&
    \le\mu\vol\cDc+\mu
      +\vol\cDc\Prob\{\vol{\{x\in\cDc\mid u_{\lambda_k}(x)\ge\mu\}}\ge\mu\}\\&
    <(2\vol\cDc+1)\mu\text.
  \end{align*}
  This proves the claimed equivalence.
\end{proof}

Since we no longer assume that the variables~$\lambda_k$
are identically distributed, the random operator family
will no longer be ergodic and the IDS may not be well-defined.
We therefore state the conclusion in the following theorem
directly on the probabilities of low eigenvalues
for finite volume restrictions of a suitable scale.
We will again consider only Mezincescu boundary conditions and
suppress the superscript~${}^M$.
\genericConstant{CexpII}%
\begin{thm}\label{thmabstract}
  Assume that $\Vper\in L^\infty(\bR^d,\bR)$,
  $\lambda_k\from\Omega\to\Omega_0$, $k\in\cL$,
  are independent random variables, that
  $u\from\Omega_0\times\bR^d\to\Ico0\infty$
  is non--degenerate, and that~$(H_\omega)$
  is defined by~\eqref{e-random-potential} and~\eqref{e-Hper-Homega}
  with $E_0:=\inf\sigma(\Hper)$.
  Then there exist $\CexpII,\delta\ge0$ and $E'>E_0$
  such that for all $E\in\Ioc{E_0}{E'}$
  \begin{equation}\label{expdecay}
    \Prob\bigl\{\omega\in\Omega:E_1(H_\omega^{L_E})\le E\bigr\}
      \le\exp\bigl(-\CexpII(E-E_0)^{-\frac d2}\bigr)\text,
  \end{equation}
  where $L_E=\floor[\big]{\sqrt{\delta/(E-E_0)}}$.
\end{thm}

\begin{remark}
  In the case of an ergodic operator family~$(H_\omega)_{\omega\in\Omega}$,
  the IDS~$N$ exists, and a direct consequence of \cref{thmabstract}
  is Lifshitz behavior at~$E_0$:
  \begin{equation*}
    \ex\Cfront,\CexpII>0,E'>E_0\colon\fa E\in\Ioc{E_0}{E'}\colon
    N(E)\le\Cfront\exp\bigl(-\CexpII(E-E_0)^{-\frac d2}\bigr)
    \text.
  \end{equation*}
\end{remark}

\begin{example}[General breather model]
  Let $\Omega_0:=\Ico0\infty$ and $u_1\in L^\infty(\bR^d)$
  be non--negative, non--vanishing and compactly supported,
  and define for $\lambda\ge0$ the breather type single site potential
  \begin{equation*}
    u\from\Ico0\infty\times\bR^d\to\Ico0\infty\,\text,\quad
    u(\lambda,x):=\begin{cases}
      u_1(x/\lambda)    &(\lambda>0)\\
      0                 &(\lambda=0)\,\text.
    \end{cases}
  \end{equation*}
  Further let $\lambda_k$, $k\in\cL$, be independent and
  identically distributed non--negative random variables.
  Then, the family~$(H_\omega)_{\omega\in\Omega}$
  is ergodic, and consequently the IDS~$N$
  is well-defined and shows Lifshitz behavior as in \cref{thmbreather}.
\end{example}

\begin{proof}[Proof of \cref{thmabstract}]
  As in \cref{subsectionProofBreather} we use the fundamental domain
  $\cDc=\M\Ioo{-\frac12}{\frac12}^d$,
  the cube $\Lambda_L=\M\Ioo{-(L+\frac12)}{L+\frac12}^d$,
  and the index set $I_L=\cL\isect\Lambda_L$ with $L\in\bN$.
  The superscript~${}^L$ denotes the restriction of operators to~$\Lambda_L$
  with Mezincescu boundary conditions, cf.\ \cref{sectionboundaryconditions}.
  Without loss of generality we assume that $E_0=\inf\sigma(\Hper)=0$.
  \par
  We will reduce the model to a random
  \Schroedinger\ operator with simplified single site potential
  satisfying properties described in~\eqref{eq:OBdA}.
  For this purpose we employ the cut--off operator
  \begin{equation}\label{abschneideoperator}
    A_\mu\from L^p(\bR^d,\Ico0\infty)\to L^p(\bR^d,\{0,\mu\})\textq,
    A_\mu u:=\mu\ifu{\{q\in\cDc\mid u(q)\ge\mu\}}
  \end{equation}
  for a $\mu>0$ such that~$u$ is $\mu$-non--degenerate.
  The non--linear operator~$A_\mu$ is weakly measurable,
  since it maps measurable functions~$u$ to measurable images
  $\tilde u:=A_\mu u=\mu\cdot(\ifu{\Ico\mu\infty}\circ u)\ifu\cDc$.
  The random \Schroedinger\ operator
  \begin{equation*}
    \tilde H_\omega
      :=-\Laplace+\Vper+\tilde V_\omega
    \qtextq{with}
    \tilde V_\omega:=\sum_{k\in\cL}\tilde u_{\lambda_k(\omega)}(\argmt-k)
  \end{equation*}
  has bounded potential and is thus well defined on the domain of~$-\Laplace$.
  \par
  Since $\tilde u\le u$ and thereby $\tilde V_\omega\le V_\omega$, we have
  $\tilde H_\omega\le H_\omega$
  for $\Prob$-almost all~$\omega\in\Omega$.
  And since Mezincescu boundary conditions depend on~$\Vper$,
  and the periodic background is the same for $(\tilde H_\omega)$ and
  $(H_\omega)$, for all the finite volume restrictions
  \begin{equation*}
    \tilde H_\omega^L\le H_\omega^L
  \end{equation*}
  for $\Prob$-a.\,a.~$\omega\in\Omega$ and all $L\in\bN$.
  It therefore suffices to show the upper bound of the theorem for
  $\tilde H_\omega$ instead of~$H_\omega$.
  By the non--degeneracy of the single site potential,
  we henceforth reduce ourselves without loss of generality to the situation%
  ~\eqref{eq:OBdA} and skip from now on the $\tilde\ $ in the notation.
  \par
  The random variables~$X_k$, $k\in\cL$, defined in~\eqref{XS},
  are now no longer identically distributed.
  However, due to~\eqref{eq:non-degenerate},
  their expectations still share a positive infimum:
  \begin{equation*}
    1\ge\E[X_k]
      =\E\int_{\supp u_{\lambda_k}}\abs{\Psi(x)}^2\,\dx
      \ge\Psi_-^2\E\vol{\supp u_{\lambda_k}}
      \ge\Psi_-^2\mu^2
      =:\beta
      >0\text.
  \end{equation*}
  Of course, the averages~$S_L$, see~\eqref{XS}, satisfy $\E[S_L]\ge\beta$, too.
  We adapt the definition of~$L_E$, substituting~$\beta$ for~$\E[X_0]$:
  \begin{equation*}
    L_E:=\floor*{\sqrt{\frac{\Cgap\beta}{8E}}}\text{, \ie\ }
    \delta:=\frac{\Cgap\beta}8\text.
  \end{equation*}
  With this change we can adapt a large part of the proof of
  \cref{thmbreather}.
  In particular, we get for all $E>0$ small enough
  \begin{equation*}
    \Prob\{\omega\colon E_1(H_\omega^{L_E})\le E\}
      \le\Prob\{S_{L_E}\le\beta/2\}\text.
  \end{equation*}
  \Cref{LDP} is replaced by the Bernstein inequality \cref{Bernstein}
  and gives
  \begin{align*}
    \Prob\{\omega\colon E_1(H_\omega^{L_E})\le E\}&
      \le\Prob\{S_{L_E}\le\tfrac\beta2\}\\&
      \le\exp\bigl(-\tfrac{\beta^2}{16}\setsize{I_{L_E}}\bigr)
      \le\exp\bigl(-\CexpII E^{-d/2}\bigr)
  \end{align*}
  with $\CexpII:=\beta^2\delta^{d/2}/16$.
\end{proof}

\begin{lemma}[cf.\ \cite{Schwarzenberger-12}]\label{Bernstein}
  Assume $\beta\in\Ioc01$ and that the independent random variables
  $X_k\from\Omega\to\Icc01$ satisfy
  $\E[X_k]\ge\beta$, $k\in\{1,\dotsc,n\}$.
  Then we have, for $S_n:=\frac1n\sum_{k=1}^nX_k$,
  \begin{equation*}
    \Prob\bigl\{S_n\le\tfrac\beta2\bigr\}
      \le\exp\bigl(-\tfrac{\beta^2}{16}n\bigr)\text.
  \end{equation*}
\end{lemma}
\begin{proof}
  Let
  \begin{equation*}
    Y_k:=X_k-\E[X_k]\in\Icc{-1}1\quad(k\in\bN)\text.
  \end{equation*}
  The exponential moments of~$Y_k$ are, for $\abs h\le1/2$, bounded by
  \begin{align*}
    \E[\exp(hY_k)]&
      =\sum_{m=0}^\infty\frac{h^m\E[Y_k^m]}{m!}
      \le1+\frac{h^2}2\sum_{m=0}^\infty\abs h^m
      \le1+h^2
      \le\e^{h^2}
    \text.
  \end{align*}
  Therefore we have by independence
  \begin{align*}
    \E\bigl[\exp\bigl(t(S_n-\E S_n)\bigr)\bigr]&
      =\prod_{k=1}^n
        \E\bigl[\exp\bigl(\tfrac tnY_k\bigr)\bigr]
      \le\exp\bigl(t^2/n\bigr)
  \end{align*}
  for all $\abs t\le\frac n2$.
  Now we employ Markov's inequality with a parameter
  $t\in\Icc0{\tfrac n2}$:
  \begin{align*}
    \Prob\{S_n\le\tfrac\beta2\}&
      \le\Prob\{S_n-\E S_n\le-\tfrac\beta2\}\\&
      =\Prob\{\exp(-t(S_n-\E S_n))\ge\exp(\beta t/2)\}\\&
      \le\exp(-\beta t/2)\E[\exp(-t(S_n-\E S_n))]\\&
      \le\exp\bigl(t^2/n-\beta t/2\bigr)\text.
  \end{align*}
  The lemma follows with the choice
  $t=\tfrac\beta4n\in\Icc0{\tfrac n2}$.
\end{proof}

We now present conditions that guarantee that the \Schroedinger\ operators
have indeed spectrum at~$E_0$ with positive probability.
\begin{lemma}\label{upperbound}
  Fix $p>\max\{2,d/2\}$.
  Let $\lambda_k\from\Omega\to\Omega_0$, $k\in\cL$,
  be a family of independent random variables on the probability space
  $(\Omega,\cA,\Prob)$ and $u\from\Omega_0\times\bR^d\to\Ico0\infty$ measurable.
  We define~$W_\omega$ by~\eqref{e-random-potential},
  assume~\eqref{eq:Wlocp} and that, for all $\kappa>0$ and all $k\in\cL$,
  we have
  \begin{equation}\label{eq:0insupp}
    \Prob\{\norm[p]{u_{\lambda_k}}\le\kappa\}>0\text.
  \end{equation}
  Then, the operator~$H_\omega$ defined by \eqref{e-Hper-Homega}
  satisfies for all $\varepsilon>0$
  \begin{equation*}
    \Prob\{\omega\in\Omega:\inf\sigma(H_\omega)\le E_0+\varepsilon\}>0\text.
  \end{equation*}
\end{lemma}

If an ergodicity assumption holds
the spectrum of~$H_\omega$ is almost surely constant,
and positive probability implies in fact full probability.

\genericConstant{psiconst}
\begin{proof}
  Fix $\varepsilon>0$.
  We use the variational characterization
  $\inf\sigma(H_\omega)=\inf\limits_{\norm[2]\psi=1}\spr\psi{H_\omega\psi}$
  and choose the vector~$\psi:=\psi_{x,L}$ from \eqref{eq:psixL}.
  We know that $\spr{\psi_{x,L}}{\Hper\psi_{x,L}}\to E_0$ as $L\to\infty$.
  Thus, there exists an~$L\in\bN$ such that
  \begin{equation*}
    \abs{\spr{\psi_{x,L}}{\Hper\psi_{x,L}}}
    <E_0+\varepsilon/3\text.
  \end{equation*}
  We obtain
  \begin{equation*}
    \abs{\spr{\psi_{x,L}}{H_\omega\psi_{x,L}}}
    \le E_0+\varepsilon/3+\abs{\spr{\psi_{x,L}}{W_\omega\psi_{x,L}}}\text.
  \end{equation*}
  We need to estimate the remaining term
  \begin{equation*}
    \abs{\spr{\psi_{x,L}}{W_\omega\psi_{x,L}}}
    \le\norm[q]{\psi_{x,L}}\norm[p]{W_\omega\psi_{x,L}}
    \le\psiconst\norm[p]{W_\omega\ifu{\Lambda_L}}
    \text,
  \end{equation*}
  where $\psiconst:=\norm[q]{\psi_{x,L}}\Psi_+/\norm[2]{\Psi\molly{x,L}}$,
  see~\eqref{eq:psixL}.
  To this end, we introduce a parameter $R\in\bN$ which we choose soon
  and split $\ifu{\Lambda_L}=\sum_{j\in I_L}\ifu{\cDc+j}$:
  \begin{equation*}
    \norm[p]{W_\omega\ifu{\Lambda_L}}
    \le\sum_{k\in I_R}\norm[p]{u_{\lambda_k(\omega)}}
    +\sum_{j\in I_L}\Bignorm[p]{\sum_{k\in\cL\setminus I_R}
      u_{\lambda_k(\omega)}(\argmt-k)\ifu{\cDc+j}}\text.
  \end{equation*}
  By~\eqref{eq:Wlocp}, we can use~$W_\omega$
  as dominating function to see that, for all $j\in\cL$,
  \begin{equation*}
    \Bignorm[p]{\sum_{k\in\cL\setminus I_R}
      u_{\lambda_k(\omega)}(\argmt-k)\ifu{\cDc+j}}
    \xto{R\to\infty}0
  \end{equation*}
  by dominated convergence.
  We choose~$R$ so large that
  \begin{equation*}
    \sum_{j\in I_L}\Bignorm[p]{\sum_{k\in\cL\setminus I_R}
      u_{\lambda_k(\omega)}(\argmt-k)\ifu{\cDc+j}}
    \le\frac\varepsilon{3\psiconst}\text.
  \end{equation*}
  All we have left to do is to make sure that
  \begin{equation*}
    \sum_{k\in I_R}\norm[p]{u_{\lambda_k(\omega)}}
    \le\frac\varepsilon{3\psiconst}\text,
  \end{equation*}
  too.
  By assumption~\eqref{eq:0insupp},
  this happens with positive probability.
\end{proof}

\section{Complementary lower bound on the IDS}
\label{s:lower-IDS-bound}
For an ergodic operator family $(H_\omega)_{\omega\in\Omega}$
the IDS~$N$ exists and in the setting of \cref{thmabstract}
exhibits a Lifshitz bound at~$E_0$.
Here we complement the statement with a lower bound.
If we strive only for the exponent in the exponent,
we have by l'H\^opital's rule
\begin{multline*}
  \limsup_{E\dnto E_0}\frac{\ln\bigl(-\ln(N(E))\bigr)}{\ln(E-E_0)}
    \le\lim_{E\dnto E_0}
      \frac{\ln\bigl(\Cexp(E-E_0)^{-\frac d2}\bigr)}
           {\ln(E-E_0)}\\
    =\lim_{E\dnto E_0}
      \frac{(E-E_0)(-\frac d2\Cexp(E-E_0)^{-\frac d2-1})}
           {\Cexp(E-E_0)^{-\frac d2}}
    =-\frac d2\text.
\end{multline*}
Under moderate additional conditions, \cref{thmlowerbound}
shows that the limit of the logarithms actually exists and equals~$-d/2$.
We use the norm
\begin{equation*}
  \norm[\ell^\infty(L^p)]{f}
  :=\sup_{k\in\cL}\norm[p,\cD+k]f\text.
\end{equation*}

\begin{thm}\label{thmlowerbound}
  Let $(H_\omega)_{\omega\in\Omega}$ be a random operator as in
  \eqref{e-random-potential}--\eqref{e-Hper-Homega}
  Assume that the random variables~$\lambda_k$, $k\in\cL$, are i.\,i.\,d.,
  and that the single site potential has a summable decay:
  \begin{equation}\label{decay}
    \exists C,\epsilon>0\colon\forall k\in\cL\colon
    \norm[p,\cD+k]{u_{\lambda_0}}
      \le C\bigl(1+\norm k\bigr)^{-(d+\epsilon)}\qtext{a.\,s.,}
  \end{equation}
  and low values of the single site potential are not too improbable:
  \begin{equation}\label{poly}
    \exists\alpha_0,\eta>0\colon\forall\alpha\in\Icc0{\alpha_0}\colon
    \Prob\bigl\{\norm[\ell^\infty(L^p)]{u_{\lambda_0}}\le\alpha\bigr\}
      \ge\alpha^\eta\text.
  \end{equation}
  Then the IDS $N\from\bR\to\bR$ exists by ergodicity, and we have
  \begin{equation}\label{liminf}
    \liminf_{E\dnto E_0}\frac{\ln(-\ln(N(E)))}{\ln(E-E_0)}
      \ge-\frac d{\min\{2,\epsilon\}}\text.
  \end{equation}
\end{thm}

\genericConstant{Cdecay}%
\begin{remark}\label{remarkdecay}
  The decay condition~\eqref{decay} has the following consequence.
  There exists $\Cdecay\ge0$
  such that for almost all $\omega\in\Omega$ and $R>0$
  \begin{align*}
    \norm[p,\cD]{W_\omega}&
      =\Bignorm[p,\cD]{\sum_{k\in\cL}u_{\lambda_k(\omega)}(\argmt-k)}
      \le\sum_{k\in\cL}\norm[p,\cD+k]{u_{\lambda_k(\omega)}}\\&
      \le\frac{\Cdecay}{\epsilon R^\epsilon}+\sum_{k\in I_R}
        \norm[\ell^\infty(L^p)]{u_{\lambda_k(\omega)}}\text.
  \end{align*}
  We can hereby control the norm of~$W_\omega$ on~$\cD$
  with the norms of the single site potentials in a box of side length~$2R+1$,
  where~$R$ is determined by the allowed error.
\end{remark}

\begin{remark}
  The requirement~\eqref{poly}
  is a quantitative version of~\eqref{supportcondition}.
  Note also, that in the case $\epsilon\ge2$
  the limit exists and is equal to~$-d/2$.
  In specific models (alloy type with long range single site potentials)
  one can derive upper bounds on the IDS which match the bound~\eqref{liminf}
  also in the case $\epsilon\in\Ioo02$,
  cf.~\cite{Pastur-77,KirschM-07,Mezincescu-87}
\end{remark}

\begin{proof}[Proof of \cref{thmlowerbound}]
The proof follows the line of \cite{KirschS-86}.
We denote the restriction of $H_\omega$ to $\Lambda_L$
with Dirichlet boundary condition by $H_\omega^{L,D}$.
By \eqref{eq:IDSsupL} and \v Ceby\v sev's inequality we see that
\begin{align}
  N(E)&
    \ge\E\bigl[N_L^D(E,H_{\argmt})\bigr]
    =\vol{\Lambda_L}^{-1}\E\bigl[n_L^D(E,H_{\argmt})\bigr]\notag\\&
    \ge\vol{\Lambda_L}^{-1}\Prob\bigl\{\omega\in\Omega\colon
      n_L^D(E,H_\omega)\ge1\bigr\}\notag\\&
    =\vol{\Lambda_L}^{-1}\Prob\bigl\{\omega\in\Omega\colon
      \inf\sigma(H_\omega^{L,D})\le E\bigr\}\text.
  \label{Cebycev}
\end{align}
For all $\varphi\in\dom(H_\omega^{L,D})\setminus\{0\}$
and a.\,a.\ $\omega\in\Omega$, we have
\begin{equation*}
  \inf\sigma(H_\omega^{L,D})
    \le\frac{\spr\varphi{H_\omega^{L,D}\varphi}}
            {\norm\varphi^2}\text.
\end{equation*}
To continue estimate \eqref{Cebycev},
we use a smoothly truncated Version $\varphi=\molly L\Psi$
of the periodic solution $\Psi\from\bR^d\to\bR$ of
$(-\Laplace+\Vper)\Psi=0$ with $\norm[2,\cD]\Psi=1$, see~\eqref{eqPsi}.
Here, $\molly L=\molly{}(\argmt/L)$ is the function from~\eqref{eq:molly},
\ie\ $\molly{}\in C^\infty(\bR^d,\Icc01)$ such that
$\molly{}\rstr{\cD}=1$ and $\supp(\molly{})\sse\Lambda_{1/2}=\M\Ioo{-1}1^d$.

\genericConstant{blubb}%
\begin{lemma}\label{lemmablubb}
  There exists a constant~$\blubb>0$
  such that for a.\,a.~$\omega\in\Omega$ and all $L\in\bN$
  \begin{equation*}
    \frac{\spr{\molly L\Psi}{H_\omega^{L,D}(\molly L\Psi)}}
         {\norm[2]{\molly L\Psi}^2}-E_0
    \le\frac{\spr{\molly L\Psi}{W_\omega\molly L\Psi}}
            {\norm[2]{\molly L\Psi}^2}
      +\frac{\blubb}{L^2}\text.
  \end{equation*}
\end{lemma}
\begin{proof}
  It suffices to show
  \begin{equation*}
    \frac{\spr{\molly L\Psi}{(\Hper^{L,D}-E_0)(\molly L\Psi)}}
         {\norm[2]{\molly L\Psi}^2}
      \le\frac{\blubb}{L^2}\text.
  \end{equation*}
  We use $(-\Laplace+\Vper)\Psi=E_0\Psi$,
  the fact that $\Psi$ is real-valued and
  $\Psi(\bR^d)=\Icc{\Psi_-}{\Psi_+}\subseteq\Ioo0\infty$,
  cf.~\eqref{eq:harnack}:
  \begin{align*}
    \delta\etdef
     \spr{\molly L\Psi}{(-\Laplace+\Vper)(\molly L\Psi)}
        -E_0\norm[2]{\molly L\Psi}^2\\&
      =-\spr{\molly L\Psi}
            {(\Laplace\molly L)\Psi+2(\nabla\molly L)\nabla\Psi}\\&
      =-\spr{\molly L\Psi^2}{\Laplace\molly L}
        -2\spr{\molly L\Psi\nabla\Psi}{\nabla\molly L}\text.
  \end{align*}
  Now partial integration gives
  \begin{align*}
    \delta&=\spr{\nabla(\molly L\Psi^2)}{\nabla\molly L}
        -2\spr{\molly L\Psi\nabla\Psi}{\nabla\molly L}\\&
      =\spr{(\nabla\molly L)\Psi^2}{\nabla\molly L}
      =\norm[2]{\Psi\nabla\molly L}^2\text.
  \end{align*}
  Finally, the rescaling produces the needed factor:
  \begin{equation*}
    \delta
    \le\norm[2]{\nabla\molly L}^2\Psi_+^2
    =L^{-2}\norm[2]{\bigl((\nabla\molly{})({}\argmt/L)\bigr)}^2\Psi_+^2
    =L^{d-2}\norm[2]{\nabla\molly{}}^2\Psi_+^2
    \text.
  \end{equation*}
  Combined with
  \begin{align*}
    \norm[2]{\molly L\Psi}^2&
      \ge\norm[2]{\molly L}^2\Psi_-^2
      =L^d\norm[2]{\molly{}}^2\Psi_-^2\text,
  \end{align*}
  we see that
  $\blubb:=\bigl(\frac{\norm[2]{\nabla\molly{}}\Psi_+}
                   {\norm[2]{\molly{}}\Psi_-}\bigr)^2$
  is a valid choice.
\end{proof}
\begin{remark}
  Note, that \cref{lemmablubb} does neither follow from nor imply
  \cref{approxEF}, since
  \begin{equation*}
    \frac{\spr{\molly L\Psi}{(\Hper^{L,D}-E_0)(\molly L\Psi)}}
         {\norm[2]{\molly L\Psi}^2}
      \le\frac{\norm[2]{(\Hper-E_0)(\molly L\Psi)}}{\norm[2]{\molly L\Psi}}
      \le\frac\perappEF L
  \end{equation*}
  is not strong enough, while
  $\norm[2]{(\Hper-E_0)(\molly L\Psi)}$
  cannot be controlled by
  $\spr{\molly L\Psi}{(\Hper^{L,D}-E_0)(\molly L\Psi)}$.
\end{remark}

\Cref{lemmablubb} and the choice $L_E:=\ceil{\sqrt{2\blubb/(E-E_0)}}$
allow us to continue the estimate~\eqref{Cebycev}:
\begin{align}
  N(E)&
    \ge\vol{\Lambda_{L_E}}^{-1}
      \Prob\bigl\{\omega\colon \inf\sigma(H_\omega^{L_E,D})\le E\bigr\}\notag\\&
    \ge\vol{\Lambda_{L_E}}^{-1}\Prob\Bigl\{\omega\colon
      \frac{\spr{\molly{L_E}\Psi}{W_\omega\molly{L_E}\Psi}}
           {\norm[2]{\molly{L_E}\Psi}^2}
      +\frac{\blubb}{L_E^2}\le E-E_0\Bigr\}\notag\\&
    \ge\vol{\Lambda_{L_E}}^{-1}\Prob\Bigl\{\omega\colon
      \frac{\norm[1,\Lambda_{L_E}]{W_\omega\Psi^2}}
           {\norm[2]{\molly{L_E}\Psi}^2}
      \le\frac{E-E_0}2\Bigr\}\text.\label{downtoW}
\end{align}
Next we break~$\Lambda_{L_E}$
into copies of the fundamental domain~$\cD=\M\Ioo{-\frac12}{\frac12}^d$,
using $p,q\in\Icc1\infty$, $\frac1p+\frac1q=1$:
\begin{align*}&
  \frac{\norm[1,\Lambda_{L_E}]{W_\omega\Psi^2}}
       {\norm[2]{\molly{L_E}\Psi}^2}
  \le\frac{\norm[q,\Lambda_{L_E}]{\Psi^2}\norm[p,\Lambda_{L_E}]{W_\omega}}
          {L^d\vol{\cD}\Psi_-^2}\\&
  =\frac{(2L+1)^d}{L^d\vol\cD\Psi_-^2}
    \biggl(\frac1{\setsize{I_{L_E}}}
      \int_{\Lambda_{L_E}}\Psi(x)^{2q}\,\dx\biggr)^{\frac1q}
    \biggl(\frac1{\setsize{I_{L_E}}}
      \int_{\Lambda_{L_E}}W_\omega(x)^p\,\dx\biggr)^{\frac1p}\\&
  \le\frac{3^d\norm[q,\cD]{\Psi^2}}{\vol\cD\Psi_-^2}
    \biggl(\frac1{\setsize{I_{L_E}}}\sum_{k\in I_{L_E}}
      \int_{\cD+k}W_\omega(x)^p\,\dx\biggr)^{\frac1p}\text,
\end{align*}
for $\Prob$-a.\,a.~$\omega\in\Omega$, $L\in\bN$.
We return to inequality~\eqref{downtoW} and use
\genericConstant{Cpsi}%
$\Cpsi:=\frac{\vol\cD\Psi_-^2}{2\cdot3^d\norm[q,\cD]{\Psi^2}}$:
\begin{equation}\label{downtonormW}
  \begin{split}
    N(E)&
      \ge\vol{\Lambda_{L_E}}^{-1}\Prob\Bigl\{\omega\colon
        \Bigl(\frac1{\setsize{I_{L_E}}}\sum_{k\in I_{L_E}}
          \norm[\cD+k]{W_\omega(x)}^p\Bigr)^{\frac1p}
        \le(E-E_0)\Cpsi\Bigr\}\\&
      \ge\vol{\Lambda_{L_E}}^{-1}\Prob\bigl\{\omega\colon
        \fa k\in I_{L_E}\colon\norm[p,\cD+k]{W_\omega}\
        \le(E-E_0)\Cpsi\bigr\}\text.
  \end{split}
\end{equation}

The next step is to reduce the condition on $W_\omega$
to a condition the single site potentials~$u_{\lambda_k(\omega)}$,
using the decay estimate from \cref{remarkdecay}.
To guarantee $\norm[p,\cD+k]{W_\omega}\le(E-E_0)\Cpsi$ for all $k\in I_{L_E}$,
it suffices to establish
\begin{equation*}
  \frac{\Cdecay}{\epsilon R^\epsilon}
    \le\frac{(E-E_0)\Cpsi}2\qtextq{and}
  \norm[\ell^\infty(L^p)]{u_{\lambda_k(\omega)}}
    \le\frac{(E-E_0)\Cpsi}{2\setsize{I_{L_E+R}}}
\end{equation*}
for all $k\in I_{L_E+R}$.
The first condition is met for all
\begin{equation*}
  R\ge R_E:=\ceil[\Big]{\Bigl(
           \frac{2\Cdecay}{(E-E_0)\Cpsi\epsilon}
         \Bigr)^{1/\epsilon}}\text.
\end{equation*}

To establish the second condition, we use the equivalence relation
``$a_E\sim b_E$ as $E\dnto E_0$''
defined as $\lim_{E\dnto E_0}a_E/b_E\in\Ioo0\infty$.
Observe that, as $E\dnto E_0$,
\begin{equation*}
  L_E+R_E
    \sim(E-E_0)^{-1/2}+(E-E_0)^{-1/\epsilon}
    \sim(E-E_0)^{-1/\bar\epsilon}\text,
\end{equation*}
where we let $\bar\epsilon:=\min\{2,\epsilon\}$.
In view of the second condition, we note that
\begin{equation*}
  \frac{(E-E_0)\Cpsi}{2\setsize{I_{L_E+R_E}}}
  =\frac{(E-E_0)\Cpsi}{2(2(L_E+R_E)+1)^d}
  \sim\frac{E-E_0}{(E-E_0)^{-d/\bar\epsilon}}
  =(E-E_0)^{1+\frac d{\bar\epsilon}}
  \text.
\end{equation*}
Therefore, the second condition with $R=R_E$ is implied by
\genericConstant{Cssp}%
\begin{equation*}
  \norm[\ell^\infty(L^p)]{u_{\lambda_k(\omega)}}
    \le\Cssp(E-E_0)^{1+(d/\bar\epsilon)}
\end{equation*}
with a suitable $\Cssp>0$ and $E-E_0$ small enough.

We continue~\eqref{downtonormW}, using the independence of~$\lambda_k$
and~\eqref{poly}:
\begin{align*}
  N(E)&
    \ge\vol{\Lambda_{L_E}}^{-1}\Prob\bigl\{\omega\colon
    \fa k\in I_{L_E}\colon\norm[p,\cD+k]{W_\omega}\le(E-E_0)\Cpsi\bigr\}\\&
    \ge\vol{\Lambda_{L_E}}^{-1}\prod\nolimits_{k\in I_{L_E+R_E}}
      \Prob\bigl\{\norm[\ell^\infty(L^p)]{u_{\lambda_k}}
        \le\Cssp(E-E_0)^{1+(d/\bar\epsilon)}\bigr\}\\&
    \ge\vol{\Lambda_{L_E}}^{-1}\bigl(\Cssp(E-E_0)^{1+(d/\bar\epsilon)}
          \bigr)^{\eta\,\setsize{I_{L_E+R_E}}}\text,
\end{align*}
By assumption, the estimate for the probability works for $E-E_0<\alpha_0$.
\genericConstant{Cilere}%
\genericConstant{Clambda}%
With $\Cilere,\Clambda>0$ such that
$\setsize{I_{L_E+R_E}}\le\Cilere(E-E_0)^{-d/\bar\epsilon}$
and $\vol{\Lambda_{L_E}}^{-1}\ge \Clambda(E-E_0)^{d/2}$ we see
\begin{equation*}
  N(E)
    \ge\Clambda(E-E_0)^{d/2}
        \bigl(\Cssp(E-E_0)^{1+\frac d{\bar\epsilon}}
          \bigr)^{\eta\Cilere(E-E_0)^{-d/\bar\epsilon}}
    \text.
\end{equation*}
We isolate the topmost exponent by
\begin{multline*}
  \frac{\ln\bigl(-\ln(N(E))\bigr)}{\ln(E-E_0)}\ge\\
    \frac{\ln\bigl(-\ln(\Clambda(E-E_0)^{d/2})
     -\eta\Cilere(E-E_0)^{d/\bar\epsilon}
        \ln\bigl(\Cssp(E-E_0)^{1+\frac d{\bar\epsilon}}\bigr)
      \bigr)}{\ln(E-E_0)}\\
    \xto{E\dnto E_0}-\frac d{\bar\epsilon}\text,
\end{multline*}
as l'H\^opital's rule shows.
Note that $\ln(E-E_0)<0$ for $E-E_0<1$.
\end{proof}

\section{Initial length scale estimate and Anderson localization}
\label{s:localization}
We state and prove an initial length scale estimate.
Such estimates serve as base in an induction scheme
called multi-scale analysis to prove localization,
cf.~\cite{Stollmann-01,GerminetK-04}.
An initial length scale estimate follows from our main result,
i.\,e.\ low probability for low eigenvalues, \cref{thmabstract},
by a Combes--Thomas estimate, \cref{combesthomas}.
We therefore state the initial length scale estimate as a Corollary to
\cref{thmabstract}.

In this section,
we have to deal with Dirichlet and Mezincescu boundary conditions,
and we denote them explicitly again.
We further denote the distance between two sets ${B},{\tilde B}\subseteq\bR^d$ by
\begin{equation*}
  \dist({B},{\tilde B}):=\inf\bigl\{\abs{x-y}\bigm|x\in {A},y\in {\tilde B}\bigr\}\text.
\end{equation*}
\genericConstant{CTone}%
\genericConstant{CTtwo}%
We will use the following Combes--Thomas estimate,
as found \eg\ in \cite[Theorem~2.4.1 and Remark~2.4.3]{Stollmann-01},
with Dirichlet boundary conditions and adapt\-ed to our needs:
\begin{lemma}\label{combesthomas}
  Let $\omega\in\Omega$ and $E^+\in\bR$.
  Let further~$p>\max\{2,d/2\}$, $A$ a symmetric, positive definite matrix,
  and $V\in L_{\mathrm{loc,unif}}^p$.
  We consider the operator~$H:=-\divergence A\grad+V$
  and let $M_V\ge\sup_{x\in\bR^d}\norm[p]{V\ifu{[x,x+1]^d}}$.
  Then there exist $\CTone=\CTone(M_V,A,E^+)$
  and $\CTtwo=\CTtwo(M_V,A,E^+)$
  such that the conditions
  \begin{enumerate}[(i),nosep]
    \item $\Lambda\subseteq\bR^d$ an open cube,
      ${B},{\tilde B}\in\cB(\Lambda)$
      s.\,t.\ $\delta:=\dist({B},{\tilde B})>0$, and
    \item $E<E_1(H_\omega^{\Lambda,D})\le E^+$,
  \end{enumerate}
  imply the estimate
  \begin{equation*}
    \norm{\ifu {B}(E-H_\omega^{\Lambda,D})^{-1}\ifu {\tilde B}}
      \le\frac{\CTone}{E_1(H_\omega^{\Lambda,D})-E}
        \exp\bigl(-\CTtwo(E_1(H_\omega^{\Lambda,D})-E)\delta\bigr)\text.
  \end{equation*}
\end{lemma}

With this tool we prove the following corollary to \cref{thmabstract}.
\genericConstant{cprime}%
\begin{cor}[Initial Length Scale Estimate]\label{ilse}
  Let $(H_\omega)_{\omega\in\Omega}$
  and $E_0:=\inf\sigma(\Hper)$ be as in \cref{thmabstract},
  $\ell,\kappa\in\bN$, and $L:=\ell^\kappa$.
  Let further ${B},{\tilde B}\in\cB(\Lambda_L)$ be as in \cref{combesthomas}
  with distance $\delta:=\dist(M^{-1}{B},M^{-1}{\tilde B})>0$.
  \par
  Then there exists $\cprime>0$ such that with $\CTone,\CTtwo$
  from \cref{combesthomas}:
  \begin{multline*}
    \Prob\bigl\{\omega\colon
      \norm{\ifu {B}(E_0+L^{-2/\kappa}-H_\omega^{L,D})^{-1}\ifu {\tilde B}}
        \le\CTone L^{2/\kappa}\exp\bigl(-\CTtwo\delta/L^{2/\kappa}\bigr)
    \bigr\}\\
      \ge1-2^dL^{(1-1/\kappa)d}\exp\bigl(-\cprime L^{d/\kappa}\bigr)\text.
  \end{multline*}
\end{cor}
\begin{remark}
  Usually one arranges~$\delta\ge L/3$.
  In this case, the upper bound in the event
  \begin{equation*}
    \CTone L^{2/\kappa}\exp\bigl(-\CTtwo\delta/L^{2/\kappa}\bigr)
      \le\CTone L^{2/\kappa}\exp\bigl(-\CTtwo L^{1-2/\kappa}/3\bigr)
  \end{equation*}
  vanishes (stretched) exponentially as soon as $\kappa>2$.
\end{remark}

\begin{proof}[Proof of \cref{ilse}]
  In a first step, in order to apply \cref{combesthomas},
  we transform~$\Lambda_L$ to a cube using $\M\from\bR^d\to\bR^d$
  via the unitary operator
  \begin{equation*}
    U\from L^2(\Lambda_L)\to L^2(\Ioo{-(L+\tfrac12)}{L+\tfrac12}^d)\textq,
    f\mapsto\sqrt{\det\M}f\circ\M\text.
  \end{equation*}
  This operator~$U$ is indeed unitary since, for all $f\in L^2(\Lambda_L)$,
  \begin{align*}
    \norm[L^2\left(\Ioo*{-(L+\frac12)}{L+\frac12}^d\right)]{Uf}^2&
    =\int_{\Ioo*{-(L+\frac12)}{L+\frac12}^d}\abs{f(\M y)}^2\abs{\det\M}\ddy\\&
    =\int_{\Lambda_L}\abs{f(x)}^2\ddx
    =\norm[L^2(\Lambda_L)]{f}^2\text.
  \end{align*}
  A straight-forward calculation reveals that $f\in H_0^2(\Lambda_L)$
  and $g:=Uf\in H_0^2(\Ioo{-(L+\frac12)}{L+\frac12}^d)$ relate via
  \begin{equation*}
    \Laplace f\circ\M
    =\divergence(\M^\top\M)^{-1}\grad g\text,
  \end{equation*}
  where~$\Laplace$ is the Laplace operator on~$\Lambda_L$,
  and~$\divergence$ and~$\grad$ are divergence and gradient on $\Ioo*{-(L+\frac12)}{L+\frac12}^d$,
  respectively.
  Reformulated with the unitary operator~$U$, this reads
  \begin{equation*}
    U\Laplace U^{-1}
    =\divergence(\M^\top\M)^{-1}\grad\text.
  \end{equation*}
  Hence, the Hamiltonian transforms as
  \begin{equation*}
    UH_\omega^{L,D}U^{-1}
    =\divergence(\M^\top\M)^{-1}\grad+\Vper\circ\M+W_\omega\circ\M\text.
  \end{equation*}
  We rewrite our term of interest with the unitary~$U$:
  \begin{align*}&
    \norm{\ifu {B}(E_0+L^{-2/\kappa}-H_\omega^{L,D})^{-1}\ifu {\tilde B}}\\&
    =\norm{\ifu B U^{-1}(E_0+L^{-2/\kappa}-UH_\omega^{L,D}U^{-1})^{-1}U\ifu {\tilde B}}\\&
    =\norm{\ifu{\M^{-1}{B}}
      (E_0+L^{-2/\kappa}-UH_\omega^{L,D}U^{-1})^{-1}\ifu{\M^{-1}{\tilde B}}}
  \end{align*}
  For all $\omega\in\Omega$ with
  $E_1(UH_\omega^{L,D}U^{-1})\ge E_0+2\ell^{-2}$,
  we have by \cref{combesthomas} with $E:=E_0+\ell^{-2}$:
  \begin{equation*}
    \norm{\ifu{\M^{-1}{B}}(E_0+\ell^{-2}-UH_\omega^{L,D}U^{-1})^{-1}\ifu{\M^{-1}{\tilde B}}}
      \le\CTone\ell^2
        \exp\bigl(-\CTtwo\delta/\ell^2\bigr)\text.
  \end{equation*}
  Using $E_1(H_\omega^{L,D})=E_1(UH_\omega^{L,D}U^{-1})$,
  we estimate
  \begin{multline*}
    q:=\\\Prob\bigl\{
      \norm{\ifu {B}(E_0+L^{-2/\kappa}-UH_\omega^{L,D}U^{-1})^{-1}\ifu {\tilde B}}
        >\CTone L^{2/\kappa}\exp\bigl(-\CTtwo\delta/L^{2/\kappa}\bigr)
      \bigr\}\\
      \le\Prob\bigl\{\omega\colon E_1(H_\omega^{L,D})\le E_0+2\ell^{-2}\bigr\}\text.
  \end{multline*}
  By \eqref{eq:D_ge_M}, $E_1(H_\omega^{L,M})\le E_1(H_\omega^{L,D})$, so
  \begin{equation*}
    q
    \le\Prob\bigl\{\omega\colon E_1(H_\omega^{L,M})\le E_0+2\ell^{-2}\bigr\}
    \text.
  \end{equation*}
  We now introduce more Mezincescu boundary conditions
  and lower the eigenvalues further:
  $H_\omega^{L,M}\ge\Directsum_{k\in I_{L,\ell}}H_\omega^{\Lambda_\ell+k,M}$,
  where $I_{L,\ell}:=\Lambda_L\isect\ell\cL$.
  Thus:
  \begin{align*}
    q&
      \le\Prob\bigl\{\omega\colon
        E_1\bigl(\Directsum_{k\in I_{L,\ell}}H_\omega^{\Lambda_\ell+k,M}\bigr)
          \le E_0+2\ell^{-2}\bigr\}\\&
    \le\sum_{k\in I_{L,\ell}}
      \Prob\{\omega\colon E_1(H_\omega^{\ell,M})\le E_0+2\ell^{-2}\}
    \text.
  \end{align*}
  Now we invoke \cref{thmabstract} and conclude
  \begin{align*}
    q&
      \le(2\ell^{(\kappa-1)})^d\exp\bigl(-c(2\ell^{-2})^{-d/2}\bigr)
      =2^dL^{(1-1/\kappa)d}\exp\bigl(-\cprime L^{d/\kappa}\bigr)\text.
      \qedhere
  \end{align*}
\end{proof}

Aiming to prove  localization we now restrict our model in order
to have a Wegner estimate at disposal:

\begin{assumption} \label{ass:for-Wegner}
  For the matrix~$\M$ defining the lattice~$\cL$ we assume:
  There is $M>0$ such that $\M=\diag(M)$ is diagonal.
  Let $u(t,x)=\mu\ifu{tA}(x)$ for $\mu>0$ and $t\in\Icc01$
  with an open, bounded, and convex set~$A$ such that the origin is contained in the closure of~$A$.
  Let us assume that the distribution measure of the i.\,i.\,d.\ random variables~$\lambda_k$
  has a bounded density~$\nu$ with support equal to $\Icc01$.
  We set $W_\omega(x)=\sum_{k\in\cL} u(\lambda_k(\omega),x)$.
  Further, we fix $G_u>0$ such that $A\subset\Lambda_{G_u/2}$.
\end{assumption}

One of the main results of \cite{TaeuferV-15,NakicTTV-18b,Taeufer-18}
states the following:

\begin{thm}[Wegner estimate]\label{thm:wegner}
If Assumption~\ref{ass:for-Wegner} holds, then for all $E_0\in\bR$,
there exist constants $C,\kappa,\varepsilon_{\max}>0$,
depending only on $d$, $M$, $A$, $E_0$, $\mu$, $G_u$, and $\norm\nu_\infty$,
such that for all $L\ge G_u$, all $E\in\bR$ and all
$\varepsilon \leq \varepsilon_{\max}$ with
$\Icc{E-\varepsilon}{E+\varepsilon}\subset\Ioc{-\infty}{E_0}$ we have
\begin{equation}
  \E\bigl[\Tr\bigl[\ifu{\Icc{E-\varepsilon}{E+\varepsilon}}
    (H_{\omega,L})\bigr]\bigr]
  \le C\varepsilon^{1/\kappa}\abs{\ln\varepsilon}^dL^d\text.
\end{equation}
\end{thm}

Combining this statement with our initial length scale estimate,
we are in position to perform multi-scale analysis to prove localization,
even strong Hilbert-Schmidt dynamical localization,
see for instance \cite{GerminetK-04,RojasMolina-12} for the definition of the last notion.
Note that the single site potentials have compact support
and thus the values of the random potential are independent
when evaluated at sufficiently distant points.
Thus we obtain:

\begin{thm}[Localization]\label{thm:localization}
  Under the assumptions above, there exists
  $E'>\min\sigma(H_\omega)$ such that
  \begin{equation*}
    \Icc{\min\sigma(H_\omega)}{E'}\isect\sigma(H_\omega)
    \ne\emptyset
    \quad\text{ and }\quad
    \Icc{\min\sigma(H_\omega)}{E'}\isect\sigma_{c}(H_\omega)
    =\emptyset
    \text.
  \end{equation*}
Furthermore, $H_\omega$ exhibits strong Hilbert-Schmidt dynamical localization in
$\Icc{\min\sigma(H_\omega)}{E'}$.
\end{thm}

\subsection*{Acknowledgements}
We thank Matthias T\"aufer to permission to use his code to produce Figure \ref{fig:breather}.
Work on this paper was partially supported by the DFG through grant no.~VE 253:9-1
entitled \emph{Random Schr{\"o}dinger operators with non-linear influence of randomness}.

\printbibliography
\end{document}